\documentclass[12pt,a4paper,twoside]{amsart}
\usepackage{amsmath}
\usepackage[all]{xy}
\usepackage{amssymb}
\usepackage{amsfonts}
\usepackage{a4}
\numberwithin{equation}{section}

\date{}

\newcommand{\coker}{{\rm cok}\,}
\newcommand{\image}{{\rm im}\,}
\renewcommand{\O}{{\mathcal{O}}}
\newcommand{\C}{{\mathbb{C}}}

\newcommand{\Hom}{\mathrm{Hom}}
\newcommand{\Trd}{\mathrm{Trd}}
\newcommand{\Nrd}{\mathrm{Nrd}}
\newcommand{\Span}{\mathrm{Span}}
\newcommand{\Stab}{\mathrm{Stab}}
\def\suml{\sum\limits}
\def\char{\text{\rm char}}

\def\G{\mathbf{G}}
\def\T{\mathbf{T}}
\def\calb{\mathcal{B}}
\newcommand\m{\mathfrak{m}}
\def\vp{\varphi}
\def\bbr{\mathbb{F}}
\def\bbf{\mathbb{R}}
\def\bbz{\mathbb{Z}}
\def\bbn{\mathbb{N}}
\def\e{\epsilon}
\def\half{\frac12}
\def\intl{\int\limits}

\def\og{\overline{\gamma}}
\def\eqadef{\mathop{=}\limits^{\rm def}}
\def\w{\omega}
\def\g{\gamma}

\newtheorem{theorem}{Theorem}[section]
\newtheorem*{theorem*}{Theorem}
\newtheorem{lemma}[theorem]{Lemma}
\newtheorem*{lemma*}{Lemma}

\newtheorem{proposition}[theorem]{Proposition}
\newtheorem{problem}[theorem]{Problem}

\newtheorem{corollary}[theorem]{Corollary}
\newtheorem{example}[theorem]{Example}

\theoremstyle{definition} 
\newtheorem{definition}[theorem]{Definition}
\newtheorem*{definition*}{Definition}
\newtheorem*{thm1*}{Theorem A1}
\newtheorem*{thm2*}{Theorem A2}
\newtheorem*{claim1*}{Claim 1}
\newtheorem*{claim2*}{Claim 2}
\newtheorem{conjecture}[theorem]{Conjecture}
{Corollary}
\newtheorem*{conjecture*}{Conjecture}

\newtheorem*{remark*}{Remark}
\newtheorem*{remarks*}{Remarks}
\newtheorem*{remark8.4}{Remark 8.4}
\def\bbz{\mathbb{Z}}
\def\bbq{\mathbb{Q}}
\def\bbf{\mathbb{F}}
\def\bbr{\mathbb{R}}

\def\bbc{\mathbb{C}}

\newcommand{\ix}{x}
\newcommand\GL{\textrm{GL}}
\newcommand\SL{\textrm{SL}}
\newcommand\SU{\textrm{SU}}
\newcommand\SO{\textrm{SO}}
\newcommand\Sp{\textrm{Sp}}

\newcommand\rk{\textrm{rk}\,}
\newcommand\Tr{\textrm{Tr}}
\newcommand{\pre}[2]{\vskip .2in\noindent\textbf{#1}. \textit{#2}\vskip .1in}
\renewcommand{\theenumi}{\roman{enumi}}

\def\calo{\mathcal{O}}

\def\calz{\mathcal{Z}}

\def\calm{\mathcal{M}}

\def\bg{\mathbf{G}}

\newcommand{\bt}{\mathbf{T}}
\def\a{\alpha}

\def\eqdef{{\buildrel \rm def \over = }}

\theoremstyle{remark}  

\newtheorem*{thmls*}{Theorem (Liebeck-Shalev)}

\begin{document}

\title{Representation Growth of Linear Groups}

\author{Michael Larsen}
\address{
Department of Mathematics\\
Indiana University\\
Bloomington, IN 47405 USA} \email{larsen@math.indiana.edu}

\author{Alexander Lubotzky}

\thanks{This research was supported by grants from the NSF and the
BSF (US-Israel Binational Science Foundation).}

\address{
Institute of Mathematics\\
Hebrew University\\
Jerusalem 91904 Israel} \email{alexlub@math.huji.ac.il }


\baselineskip 13pt
\begin{abstract}
Let $\Gamma$ be a group and $r_n(\Gamma)$  the number of its
$n$-dimensional irreducible complex representations.  We define
and study the associated representation zeta function
$\calz_\Gamma(s) = \suml^\infty_{n=1} r_n(\Gamma)n^{-s}$.  When
$\Gamma$ is an arithmetic group satisfying the congruence subgroup
property then $\calz_\Gamma(s)$ has an ``Euler factorization". The
``factor at infinity" is sometimes called
the ``Witten zeta function" counting the rational representations of
an algebraic group. For these we determine precisely the abscissa
of convergence. The local factor at a finite place  counts the
finite representations of suitable open subgroups $U$ of the
associated simple group $G$ over the associated local field $K$.
Here we show a surprising dichotomy:  if $G(K)$ is compact (i.e.
$G$ anisotropic over $K$) the abscissa of convergence goes to
0 when $\dim G$ goes to infinity, but for isotropic groups it
is bounded away from $0$.
As a consequence, there is an unconditional positive
lower bound for the abscissa for arbitrary finitely generated linear groups.
We end with some observations
and conjectures regarding the global abscissa.
\end{abstract}
\maketitle

\section{Introduction}

Let $\Gamma$ be a finitely generated group and let $s_n(\Gamma)$
denote the number of its  subgroups of index at most $n$.  The
behavior of the sequence $\{s_n (\Gamma)\}^\infty_{n = 1}$ and its
relation to the algebraic structure of $\Gamma$  has been the
focus of intensive research over the last two decades under the
rubric ``Subgroup Growth"---see \cite{LS} and the references
therein.

Counting subgroups is essentially the same as counting permutation
representations.  In this paper we take a wider perspective:  we
count linear representations. So, let $r_n(\Gamma)$ be the number
of $n$-dimensional irreducible complex representations of
$\Gamma$.  This number is not necessarily finite, in general (see
\S 4  below) but we consider only groups $\Gamma$ for which this is
the case.  In particular, it is so for the interesting family of irreducible
lattices in higher-rank semisimple groups which will be our main
cases of interest.  By Margulis' arithmeticity theorem \cite[p.~2]{Ma},
any such $\Gamma$ is commensurable to $\bg(\calo_S)$ where
$\bg$ is an $\calo_S$-subgroup scheme of $\GL_d$ with absolutely
almost simple generic fiber.  Here $k$ is a global field, $\calo$
its ring of integers, $S$ a finite subset of $V$, the set of
valuations of $k$, containing $V_\infty$, the set of
archimedean valuations, and $\calo_S$ the ring of $S$-integers.

The (finite dimensional complex) representation theory of $\Gamma$
is captured by the group $A(\Gamma)$, the proalgebraic completion
of $\Gamma$.  In \S 2, we present some background and basic
results on $A(\Gamma)$.  If $\Gamma = \bg(\calo_S)$ as before and
if in addition $\Gamma$ satisfies the congruence subgroup property
(CSP, for short), i.e.
$$C (\Gamma):= \ker\big(\widehat{\bg(\calo_S)} \to \bg(\hat\calo_S)\big)$$
is finite, then $A(\Gamma)$ can be described quite precisely:
\begin{proposition}  Let $\Gamma = \bg(\calo_S)$ as before and assume
$\Gamma$ has the congruence subgroup property. Then $A(\Gamma)$
has a finite normal subgroup $C$ isomorphic to $C(\Gamma) = \ker
(\widehat{\bg(\calo_S}) \to \bg(\hat{\calo_S}))$ such that
$$
A(\Gamma)/C \cong \bg(\bbc)^r\times \prod\limits_{v \in
V_f\backslash S} \bg(\calo_\nu) $$ where $r$ is the number of
archimedean valuations of $k, V_f = V\backslash V_\infty$, and
$\calo_v$ is the completion of $\calo$ with respect to a finite
valuation $\nu$.
\end{proposition}

\bigskip

Note that $A(\Gamma)$ is a direct product of its identity
component $\bg(\bbc)^r$ and $\hat \Gamma$, the profinite
completion of $\Gamma$.  Moreover, $\Gamma $ is embedded in
$\bg(\bbc)^r$ via the diagonal map: $\Gamma = \bg(\calo_S)
\to\prod\limits_{v\in V_\infty} \bg(k_v) \le \bg(\bbc)^r$.

Implicit in the Proposition is the fact that the CSP implies
super-rigidity: If $\rho $ is a finite dimensional complex
representation of $\Gamma$ then it can be extended on some finite
index subgroup to a rational  representation of $\bg(\bbc)^r$.

Recall now that Serre's conjecture \cite{Se} asserts that if $G$ is simply connected
and $\suml_{\nu\in S} \rk_{k_v} (\G)\ge 2$ then $\Gamma$ has the CSP.
In most cases this has been proved (see \cite[\S9.5]{PR}
and the references therein).  Moreover, in \cite{LuMr} it is shown that
if $\Gamma$ has the CSP then $r_n(\Gamma)$ is polynomially bounded
 when $n\to\infty$.  (It is further shown that if $\char (k) = 0$
 this property is equivalent to the CSP and it is conjectured that the same is true
 in general).  Let us now define:
 \begin{definition}
The representation-zeta function of $\Gamma$ is defined to be
$$\calz_\Gamma(s) = \suml^\infty_{n=1} r_n(\Gamma)n^{-s}$$

Its abscissa of convergence is:
$$
\rho(\Gamma) =\mathop{\lim\sup}\limits_{n\to\infty} \; \;
\frac{\log  R_n(\Gamma)}{\log n}
$$
where $ R_n(\Gamma) =\suml^n_{i = 1} r_i(\Gamma)$, the number of
irreducible representations of degree at most $n$.

Our main goal in this paper is to initiate the study of
representation zeta functions of arithmetic groups $\Gamma$, in
analogy with the theory of subgroup zeta functions  of nilpotent
groups (cf. \cite{DG} and \cite[Chapters 15 and 16]{LS}).

\medskip

So, if $\Gamma$ has the CSP then $\rho(\Gamma) < \infty$. The
study of $\rho(\Gamma)$ will be one of our main goals. This makes
sense for any finitely generated group.  If $R_n(\Gamma)$ is not
polynomially bounded (in particular, if $R_n(\Gamma)$ is infinite for some
$n$) we simply write $\rho(\Gamma) = \infty$.
\end{definition}

Assume for simplicity now that $\Gamma$ has the CSP and the congruence
kernel $C(\Gamma)$ is trivial.  Proposition 1.1  implies now the
important ``Euler factorization" of $\calz_\Gamma (s)$.

\begin{proposition} If $\Gamma =\G (\calo_S)$, $\Gamma$ has the CSP and
$C(\Gamma) = \{ e \}$ then
$$
\calz_\Gamma (s) = \left(\calz_{\bg(\bbc)} (s)\right)^r \times
\prod\limits_{v \in V_f\backslash S} \calz_{\bg(\calo_v)} (s)
$$
\end{proposition}

Of course, here we are using  the notation $\calz_H(s)$ for
groups $H$ which are not discrete.
When $H$ is a profinite group (resp. the group of real or complex points of an algebraic group),
we count only continuous (resp. rational) representations.

\bigskip

A concrete example to think about is $\Gamma = \SL_3(\bbz)$ for
which
$$
\calz_{\SL_3(\bbz)} (s) = \calz_{\SL_3(\bbc)} (s) \times
\prod\limits_p \calz_{\SL_3(\bbz_p)}(s).
$$

So, we have an Euler factorization with $p$-adic factors as well
as a factor at infinity.  We note here that the $p$th local factor
is not quite a power series in $p^{-s}$, i.e., it does not count
the irreducible representations of $p$-power degrees, but this is
not too far from the truth as $\SL_3(\bbz_p)$ is a virtually
pro-$p$ group (see \S 4 and \S 6).  Anyway, we can define
$\rho_\infty (\Gamma)$ to be the abscissa of convergence of the
identity component of $A(\Gamma)$ i.e. of $\bg(\bbc)^r$.  But as
$\calz_{\bg(\bbc)^r} (s) = (\calz_{\bg(\bbc)}(s))^r$ this is equal
to $\rho(\bg(\bbc))$. The factor of infinite $\calz_{\bg(\bbc)}
(s)$, the so-called ``Witten zeta function" is discussed in \S5
below.

Similarly for every $v \in V_f$ we have $\rho_v (\Gamma) =
\rho(\bg(\calo_v))$, the $v$-local abscissa of convergence.

\pre{Theorem \ref{Archimedean}}
{For $\bg$ as before,
\[
\rho\left(\bg(\bbc)\right) \, = \, \frac r\kappa
\]
where $r = \rk \bg = $(absolute) rank of $\bg$ and $\kappa =
|\Phi^+|$ where $\Phi^+$ is the set of the positive roots in the
absolute root system associated to $\bg$.}

Note that $\kappa = |\Phi^+| = \half (\dim\bg - \rk \bg)$ and
$\frac{r}{\kappa} = \frac 2h$ where $h$ is the Coxeter number of
$\Phi$.

The expression $\frac r\kappa$ has already appeared in an
analogous context in the work of Liebeck and Shalev:
\begin{theorem}[Liebeck-Shalev \cite{LiSh2}] Let $\bg$
be a Chevalley group scheme over $\bbz$.  Then
$$\lim\mathop{\sup}\limits_{n, q \to \infty} \; \frac{\log
r_n\left(\bg(\bbf_q)\right)}{\log n} \, = \, \frac r\kappa
$$
\end{theorem}

\bigskip

For $\bg (\calo_v)$ as above, we prove:

\pre{Proposition \ref{ssimp-bound}} {$\rho\left(\bg(\calo_v)\right) \ge \frac r \kappa$.}

In the anisotropic case in characteristic zero, we can prove equality.

\pre{Theorem \ref{SL1}}
{If $\bg(K) = \SL_1 (D)$ where $D$ is a division algebra of degree
$d$ over a local field $K$ of characteristic $0$, then $\bg(K)$ is
compact virtually pro-$p$ group and
$$\rho\left(\bg(K)\right) = \frac r\kappa = \frac 2 d.$$}

Jaikin-Zapirain \cite{Ja2} computed the $v$-adic local zeta function of
$\SL_2(\calo_v)$.  From his result one sees that $\rho = 1 = \frac
r \kappa$ for all such groups.

All these examples suggested to us that
$\rho\left(\bg(\calo_v)\right)$ would always be equal to $\frac r\kappa$.
The truth, however, is quite different:
\pre{Theorem \ref{Isotropic}}
{If $K$
is a non-archimedean local field, $\bg$ an isotropic simple
$K$-group, and $U$ an open compact subgroup of $\bg(K)$, then
$\rho(U) \ge \frac 1{15}$.}

We remark that $\frac 1{15}$ is probably not the best possible
constant.  It is dictated by the fact that for $E_8$ (and for other
exceptional groups with smaller Coxeter number),
we do not know how to improve on the bound of Proposition~\ref{ssimp-bound}.
We note also that for such non-archimedean local fields $K$,
the only anisotropic groups are those of the
type $\bg(K) = \SL_1(D)$ described in Theorem~\ref{SL1}.  For these,
$\frac r\kappa$ goes to zero when $\dim D$ goes to
infinity.  So Theorems~\ref{SL1} and \ref{Isotropic} give a dichotomy between
isotropic and anisotropic groups.  The latter case we understand well; we can estimate the
number of representations of given degree by counting coadjoint orbits.  In the former case,
there is a distinction between $\bg(K)$-orbits and $\bg(\calo_v)$-orbits which appears to be
controlled by the rate of growth of balls in the Bruhat-Tits building of $\bg$ over $K$.
When this rate of growth is high enough, it dominates the estimates of representation growth.
Unfortunately, we still do not know how to compute the precise rates of growth in this case.
(See \S 11  below for more on this point of view, which suggested the computations of \S8 but is not made explicit there.)

An unexpected consequence of Theorem~\ref{Isotropic} is

\pre{Theorem \ref{Linear-bound}}
{If $\Gamma$ be a finitely generated group with some linear representation
$\varphi: \Gamma \to \GL_n(F)$, with $F$ a field, such that
$\varphi(\Gamma)$ is infinite (e.g. $\Gamma$ an infinite linear
group) then $\rho(\Gamma) \ge \frac 1{15}$.}

On the other hand, we show in \S9 that there exist infinite, finitely
generated, residually finite groups $\Gamma$ with $\rho(\Gamma) = 0$.

\medskip

In \S10, we analyze $\rho(\Gamma)$ for arithmetic lattices in
semisimple groups of a very special type, namely, powers of $\SL_2$.
These are very special cases (and, as we saw above, in this problem
special cases can be quite misleading.) We still believe in the conjecture
these examples suggest:

\begin{conjecture}
\label{LatticeComparison}
{\it Let $H$ be a higher-rank semisimple group (i.e.
$H$ is a product  $\prod\limits^\ell_{i = 1} G_i(K_i)$ where each $K_i$
is a local field, each $G_i$ is an absolutely almost simple
$K_i$-group, and we have $\suml^\ell_{i=1} \rk_{K_i}(G_i) \ge
2)$. Then for any two irreducible lattices $\Gamma_1$ and
$\Gamma_2$ in $H$, $\rho(\Gamma_1) = \rho(\Gamma_2)$.}
\end{conjecture}

\medskip

This last conjecture should be compared with \cite[Theorem 11]{LuNi}
concerning the growth of $s_n(\Gamma)$, the number of subgroups of
index less than or equal to $n$, in an irreducible lattice of a
higher rank semisimple group:
\begin{theorem}[Lubotzky-Nikolov \cite{LuNi}]
\label{subgroup-growth}
Let $H$ be a higher-rank semisimple group.  Assuming the GRH
(generalized Riemann hypothesis) and Serre's conjecture, for
every irreducible lattice $\Gamma$ in $H$,  the limit
$\lim\limits_{n\to \infty} \; \frac{\log s_n (\Gamma)}{(\log
n)^2/\log\log n}$ exists and equals $\tau (H)$,
an invariant of $H$ which is given explicitly in \cite{LuNi}.
\end{theorem}

See \cite{LuNi} for further information,
including many cases for which the theorem is
proved unconditionally.

Theorem~\ref{subgroup-growth} says that the subgroup growth (i.e.,
the permutation representation rate of growth) is very similar for
different irreducible lattices in $H$.
Conjecture~\ref{LatticeComparison} makes a similar statement
regarding their finite dimensional complex representations.

There is still a significant difference.  While in \cite{LuNi} a
precise formula is given for $\tau(H)$, so far, we do not
even have a guess what will be the common value predicted by
Conjecture~\ref{LatticeComparison}.  It seems likely that one needs first to understand the
local abscissas of convergence, but even knowing them in full does
not necessarily give the global abscissa.

\medskip

The paper is organized as follows:  in \S2 we describe $A(\Gamma)$,
the proalgebraic completion, and $B(\Gamma)$, the Bohr
compactification, of a higher rank arithmetic
group $\Gamma$.  In \S3 and \S4 we show how the congruence subgroup
property gives the precise structure of $A(\Gamma)$ and out of
this an Euler factorization is deduced for $\calz_\Gamma(s)$.  The
factor at infinity is studied in \S5 where a precise formula is
given for its abscissa of convergence (Theorem~\ref{Archimedean}).  The finite
local factors are studied in \S6 (generalities), \S7 (the anisotropic case---Theorem~\ref{SL1}), and in \S8 (the isotropic case---Theorem~\ref{Isotropic}).  The
applications to discrete groups are derived in \S9.  In \S10, we
give some evidence for Conjecture~\ref{LatticeComparison}.  We end in \S11
with remarks and suggestions for further research.  It seems that
our results reveal only the tip of the iceberg of $\calz_\Gamma
(s)$.

\bigskip

\noindent{\bf Notations and Conventions}

\medskip

In this paper representations always mean complex finite
dimensional representations.

We study representation theory of various discrete groups $\Gamma$
which are always assumed to be finitely generated.

\section{The proalgebraic completion and Bohr compactification of arithmetic groups}

Let $\Gamma$ be a finitely generated group.  A useful tool for
studying the finite dimensional representation theory of $\Gamma$
over $\bbc$ is the proalgebraic completion $A(\Gamma)$ of
$\Gamma$, known also as the Hochschild-Mostow group of $\Gamma$.
(See \cite{HM}, \cite{LuMg} and \cite{BLMM} for a systematic
description.)  The group $A(\Gamma)$ together with the structure
homomorphism
\begin{equation}\tag{2.1}
i : \Gamma\to A(\Gamma)
\end{equation}
is uniquely characterized by the following property: For every
representation $\rho$ of $\Gamma$ there is a unique rational
representation $\bar\rho $ of $A(\Gamma)$ such that $\bar \rho
\circ i = \rho$.

This implies that the representation theory of $\Gamma$ is
equivalent to the rational representation theory of $A(\Gamma)$.
The image $\bar \rho(A(\Gamma))$ is always the Zariski closure of
$\rho(\Gamma)$ and in fact, $A(\Gamma)$ is the inverse limit of
these closures over all representations of $\Gamma$.  In
particular, $A(\Gamma)$ is mapped onto the profinite completion
$\hat\Gamma$ of $\Gamma$ (which can be thought  as the inverse
limit over the representations with finite image).  The kernel
$A(\Gamma)^\circ$ of the exact sequence:
\begin{equation}\tag{2.2}
1 \to A(\Gamma)^\circ \to A(\Gamma) \to \hat \Gamma \to 1
\end{equation}
is the connected component of $A(\Gamma)$.  It is a simply
connected proaffine algebraic group \cite[Theorem 1]{BLMM}

The group $\Gamma$ is called \emph{super-rigid}  if $A(\Gamma)$ is
finite dimensional (i.e., $A(\Gamma)^\circ$ is finite
dimensional).  It is shown in \cite[Theorem 5]{BLMM} that if $\Gamma$
is linear over $\bbc$ and super-rigid then it has a finite index
normal subgroup $\Gamma_0$ such that $A(\Gamma_0)\simeq
A(\Gamma_0)^\circ \times \hat \Gamma_0$.

It can be easily seen that $\Gamma_0$ can be chosen so that
$\Gamma_0 \to A(\Gamma_0)^\circ$ is injective and every
representation of $\Gamma$ can be extended, on a finite index
subgroup $\Gamma_1$ of $\Gamma_0$ (and therefore of $\Gamma)$ to a rational
representation of $A(\Gamma_0)^\circ = A (\Gamma)^\circ$. (Note,
that for a finite dimensional rational representation of
$A(\Gamma_0)$, the image of $\hat \Gamma_0$ is finite).  So,
super-rigidity for a linear group $\Gamma$ implies, and in fact is
equivalent, to the existence of a finite dimensional connected,
simply connected, algebraic group $G$ containing a finite index
subgroup $\Gamma_0$ of $\Gamma$, such that every representation of
$\Gamma$ can be extended to $G$ on some finite index subgroup of
$\Gamma_0$.

As is well known, Margulis' super-rigidity theorem (\cite[p.~2]{Ma}
says that irreducible lattices
$\Gamma$ in higher rank semisimple groups $H$ are super-rigid.
(This has now been supplemented (\cite{Co}, \cite{GS})
for lattices in $\Sp(n, 1),\;n \ge 1$, and $F^{(-20)}_4 $.)
Margulis' arithmeticity theorem \cite[p.~2]{Ma} (which is deduced
from the super-rigidity) says that every such $\Gamma$  is $(S-)$
arithmetic.

Let us now spell out the precise meaning of this regarding
$A(\Gamma)$:

So let $H$ be a semisimple (locally compact) group.  By this we
mean
\begin{equation}\tag{2.3}
H = \prod\limits^\ell_{i = 1} G_i (K_i)
\end{equation}
where each  $K_i$ is a  local field and $G_i$ is an absolutely
almost simple group defined over $K_i$.  We assume that no
$G_i(K_i)$ is compact, i.e., $\rk_{K_i} (G_i) \ge 1$.

If $\suml^\ell_{i=1} \rk_{K_i} (G_i) \ge 2$; or if $\ell = 1$,
$K_1= \bbr$, and $G_1(\bbr)$ is locally isomorphic to one of the
real rank one groups $\Sp(n, 1)$ or $F_4^{\,(-20)}$, then every
irreducible lattice of $H$ is arithmetic.  This means that there
exists a global field $k$, a finite set of valuations $S$ of $k$
containing all the archimedean ones, with $\calo_S = \{ x \in
k\mid v(x) \ge 0\; \; \; \forall v \notin S\}$, and a group
scheme of finite type $\bg/\calo_S$ whose generic fiber is
connected, simply-connected and semisimple, with a continuous map
$\psi:\prod\limits_{v \in S} $ $\bg(k_v) \to H$ whose kernel and
cokernel are compact and such that $\psi\left(\G (\calo_S)\right)$
is commensurable to $\Gamma$. (We note that the scheme can
be chosen to be flat -- see \cite[1.1]{BLR}.)

This in particular implies that if an irreducible lattice in $H$
exists, then all the fields  $K_i$ are of the same characteristic,
and all the algebraic groups $G_i$ are forms of the same group. It
also says that such a lattice $\Gamma$ is isomorphic, up to finite
index, to $\bg(\calo_S)$.

\bigskip

We can now describe the pro-algebraic completion of $\G
(\calo_S)$:
\begin{theorem}
\label{Proalg-completion} With the notation of $\bg(\calo_S)$ as
above (including the assumption $\suml_{v \in S} \rk_{k_v} (\bg)
\ge 2$; or $\ell =1$, $K_1 = \bbr$, and $G_1 (K_1)$ is either $\Sp
(n, 1)$ or $F_4^{(-20)}$) we have
\begin{equation}\tag{2.4}
A(\bg(\calo_S)) = \bg(\bbc)^{\#S_\infty} \times \widehat{\bg
(\calo_S)}
\end{equation}
where $S_\infty$ is the set of archimedean valuations of $k$.
\end{theorem}
\begin{proof}
If $k$ is of positive characteristic then by \cite[Theorem 3, p.3]{Ma},
$A(\bg(\calo_S)) = \widehat{\bg(\calo_S)}$ and we are done. Assume
$\char(k) = 0$ and then by the same theorem, for every complex
representation  of $\Gamma = \bg(\calo_S)$,
the identity  component $\overline{\Gamma}^\circ$ of
the Zariski closure of $\Gamma$) is semisimple. By \cite[Theorem 5, p. 5]{Ma} every such representation  of $\Gamma$, or of a finite index
subgroup thereof, into a simple algebraic $\bbc$-group is obtained
(up to finite index subgroup) by embedding $\calo_S$ into $\bbc$
and then composing with an
algebraic representation of $\bg(\bbc)$.

We can therefore deduce that with $\Gamma$ embedded diagonally in
$M = \prod\limits_{v\in S_\infty} \bg (\bbc)$, every complex
representation of $\Gamma$ can be extended, on a finite index
subgroup of $\Gamma$, to a representation of $M$.  This proves
that $A(\Gamma)^\circ\cong M$.

We have a direct product decomposition $A(\Gamma) = A
(\Gamma)^\circ \times \hat \Gamma$ since $\Gamma$ is indeed
densely embedded in $M = A(\Gamma)^\circ$ and hence there is a map
$A(\Gamma)\twoheadrightarrow A(\Gamma)^\circ$.
\end{proof}

So super-rigidity gives the complete description of
$A(\Gamma)^\circ$. We should now concentrate on $\hat\Gamma =
\widehat{\bg(\calo_S)}$. Here we need the congruence subgroup
property to be discussed in the next section.  We mention here in
passing that super-rigidity also gives the complete description of
the Bohr compactification of $\Gamma$.  Let us first recall:
\begin{definition} For a finitely generated group $\Gamma$ we
denote by $B(\Gamma)$   its Bohr compactification. This is a
compact group  together with a homomorphism $j : \Gamma \to B(
\Gamma)$ with the following universal property: If $\varphi$ is a
homomorphism of $\Gamma$ into some compact group $K$, there exists
a unique continuous extension $\tilde\varphi: B(\Gamma) \to K$
with $\tilde \varphi \circ j = \varphi$.

\bigskip

The existence of such $B(\Gamma)$ (and $j$) is easy to establish:
Let $\{ C_\a, \psi_\a\}$ be the family of all possible
homomorphisms $\psi_\a: \Gamma \to K_\a$ where $K_\a$ is a compact
group.  Take $C=\prod\limits_\a K_\a$,  and then  $B(\Gamma)$ is
the closure of the image of $\Gamma$ in $C$ under the diagonal map
$\gamma \to (\psi_\a(\gamma))_\a$ for $\gamma \in\Gamma$. The Bohr
compactification is of importance in the theory of almost periodic
functions (\cite[Chapter~VII]{Cd}).

\end{definition}
\begin{proposition} Let $\Gamma = \bg(\calo_S)$ be as in Theorem~\ref{Proalg-completion}.
Then
\[
B(\Gamma) = \prod\limits_{\sigma \in T} {}^\sigma\bg(\bbr)
\times\widehat{\bg(\calo_S)}
\]
where $T$ is the set of all real embeddings of $k$ for which
${}^\sigma\bg(\bbr)$ is compact, where ${}^\sigma\bg =
\bg\times_{\sigma} \bbr$.

Note that $T$ can be considered as a subset of $S_\infty$.
\end{proposition}
\begin{proof}
By the Peter-Weyl theorem every compact group is an inverse limit
of finite dimensional compact Lie groups.  Let $L =
\prod\limits_{\sigma \in T} {}^\sigma\bg ( \bbr)$.
To prove that
$B(\Gamma)^\circ = L$ means proving that if $\psi\colon\Gamma \to K$ is a
homomorphism of $\Gamma$ into a dense subgroup of a
compact Lie group $K$, then $\psi$ can be extended, up to a finite
index subgroup, to a continuous homomorphism from $L$ to $K$.

As $K$ is compact, its identity component is the group of real
points of a real connected algebraic group, $K^\circ = H(\bbr)$.
Again, as in the proof of Theorem~\ref{Proalg-completion}, if
$\char (k) > 0$, then $\psi$ has finite image and $B(\Gamma) =
\hat \Gamma$.  If $\char (k) = 0$, $H$ is semisimple and each one
of its almost simple factors is absolutely almost simple
over $\bbr$ (otherwise,
it would be a restriction of scalars of a complex group and hence
not compact). We can use \cite[Theorem 5, p. 5]{Ma} again to
deduce that the connected component of $B(\Gamma)$ is indeed $L$.
As before, it is a direct factor since we have a dense map from
$\Gamma$ to $L$.
\end{proof}

\section{The congruence subgroup property}

We continue with the notation of the previous section. So $\bg$
is a group scheme of finite type over $\calo_S$,  the ring of $S$-integers
in a global field $k$, whose generic fiber is connected,  simply
connected, and absolutely almost simple, and $\Gamma =
\bg(\calo_S)$.
\begin{definition} The group $\Gamma$ is said to have the
\emph{congruence subgroup property} (CSP for short) if $\ker
(\widehat{\bg(\calo_S)} \mathop{\rightarrow}\limits^\pi
\bg(\hat{\calo_S}))$ is finite.

\bigskip

Now by the strong approximation theorem (cf. \cite[Theorem~7.12]{PR}
and \cite{Pra})
$\pi$ is onto.  Moreover, $\bg(\hat\calo_S) =
\prod\limits_{v\notin S}\bg( \calo_v)$.  Note, that if $\Gamma$
has the CSP then by replacing $\Gamma$ with a suitable finite index
subgroup $\Gamma_0$, we have $\hat\Gamma_0 = \prod\limits_{v\notin
S} L_v$, where $L_v$ is open in $\bg(\calo_v)$ for every $v$ and
equal to it for almost every $v$.

Before continuing, let us recall (see \cite[\S16]{BMS}, \cite[\S2.7]{Se}, and
\cite[Theorem~7.2]{Ra}) that
the CSP implies super-rigidity.  In our language this means
\end{definition}
\begin{theorem}  If $\Gamma = \bg(\calo_S)$ has the CSP then
$A(\Gamma)^\circ$ is finite dimensional.
\end{theorem}

\medskip

\noindent{\bf Sketch of proof:}  First consider a representation
$\rho: \Gamma \to \GL_n(\bbq)$.  Unless $\Gamma$ is a lattice in a
rank one group over a positive characteristic field, in which case
$\Gamma$ does not have the CSP (see \cite[Theorem~D]{Lu2}), $\Gamma$ is
finitely generated and hence the entries of $\rho(\Gamma)$ are
$p$-adic integers for almost every prime $p$.  Choose such a prime
$p$  (which is not $\char (k)$). Thus we have a
representation into $\GL_n(\bbz_p)$.  This last group has a finite
index torsion-free pro-$p$ subgroup $H$.  Now, if $\Gamma$ has
CSP, then after passing to a finite index subgroup $\Gamma_0$ of
$\Gamma$, $\hat\Gamma_0 = \prod\limits_{v \notin S} L_v$ where
$L_v$ is open in $\bg(\calo_v)$.  If $\char (k) = \ell > 0$ then
$L_v$ is a virtually pro-$\ell$ group and so its image in $H$ is
finite and hence trivial.  This proves that $\rho(\Gamma)$ was
finite to start with.  If $\char (k) = 0$ then for every $v $
which does not lie over $p$, $\rho(L_v)$ is finite and again
trivial.  So we get a map from $\prod\limits_{v|p} L_v$ to
$\GL_n(\bbz_p)$. This is a map between two $p$-adic analytic
virtually pro-$p$ groups, which must be analytic and in fact
algebraic as $\bg$ is semisimple.  Thus altogether, $\rho$ can be
extended, on a finite index subgroup, to an algebraic
representation of $\bg$.

The above proof works word for word also for representations over
number fields and hence also with regard to  representations into
$\GL_n( \overline{\bbq})$, where $\overline \bbq$ is an algebraic
closure of $\bbq$.  This implies in particular that $\Gamma$ has
only finitely many irreducible $n$-dimensional $\overline
\bbq$-representations.  Indeed, if $\Gamma$ has the CSP then it has
FAb, i.e., $|\Delta/(\Delta, \Delta]| < \infty$ for every finite
index subgroup $\Delta$ of $\Gamma$.  It follows now from Jordan's
Theorem (cf. \cite[p.~376]{LS}; see also \cite[Cor.~8]{BLMM})
that $\Gamma$ has only finitely many
$n$-dimensional representations with finite image.  The same
applies also to algebraic representations of $\bg$. By the
Nullstellensatz the same applies to representations over $\bbc$.
So the character variety is finite (see [LuMg]) and all the
representations can be conjugated into $\GL_n(\overline
\bbq)$.\hfill \qed

\medskip

Note also that if $\Gamma$ has the CSP then by replacing
$\Gamma$ by a suitable finite index $\Gamma_0$ as before,
$\hat\Gamma_0 = \prod\limits_v L_v$, and combining this with
the proof of Theorem~\ref{Proalg-completion}
above we get:
\begin{theorem} If $\Gamma = \bg (\calo_S)$ has the CSP then for a
suitable finite index subgroup $\Gamma_0$ of $\Gamma$ (with
$\Gamma_0 = \Gamma$ if $\ker (\widehat{\bg(\calo_S)} \to
\bg(\hat\calo_S)) = \{ e \})$
\[
A(\Gamma_0) = \G(\bbc)^{\# S_\infty} \times \prod\limits_{v \notin
S} L_v
\]
where $L_v$ is open in $\G (\calo_v)$ and equal to it for almost
all $v$.
\end{theorem}
Finally, we mention the main result of \cite{LuMr}:

\begin{theorem}[Lubotzky-Martin \cite{LuMr}]
\label{Lubotzky-Martin}
If $\Gamma = \G  (\calo_S)$ has the CSP then
$r_n(\Gamma)$ is polynomially bounded.  If $\char (k) = 0$ then
the converse is also true.
\end{theorem}
It is conjectured in [LuMr] that the converse also holds if
$\char(k) > 0$ and some steps in this direction are taken there.

\section{The representation zeta function}

Let $\Gamma$ be a finitely generated group and $r_n (\Gamma)$  the
number of its $n$-dimensional irreducible complex representations.
This may not be a finite number.  Similarly, denote by $\hat
r_n(\Gamma)$ the number of $n$-dimensional irreducible
representations of $\Gamma$ with finite image.

\begin{proposition} (\cite[Proposition~2]{BLMM}) We have $\hat r_n(\Gamma) < \infty$  for
every $n$ if and only if  $\Gamma$ has (FAb), i.e.
$|\Delta/[\Delta, \Delta]| < \infty$ for every finite index
subgroup $\Delta$ of $\Gamma$.
\end{proposition}
On the other hand there is no known intrinsic characterization of
groups $\Gamma$ for which $r_n(\Gamma) < \infty$ for every $n$.
Such a group is called \emph{rigid}.
\begin{problem} Characterize rigid groups.
\end{problem}

Anyway, we  assume from now on  that $\Gamma$ is rigid and define:
\begin{definition} (a) \ The representation zeta function of $\Gamma$ is
$$\calz_\Gamma (s) = \suml^\infty_{n=1} r_n (\Gamma) n^{-s},$$
and the finite-representation zeta  function is
$$\hat\calz_\Gamma (s)= \suml^\infty_{n=1} \hat r_n (\Gamma) n^{-s}.$$

\noindent (b) \ Let $\rho(\Gamma) = \overline{\lim} \frac{\log
R_n(\Gamma)}{\log n}$ where $R_n(\Gamma) = \suml^n_{i=1}
r_i(\Gamma)$. It is called the \emph{abscissa of convergence} of
$\calz_\Gamma (s)$.
\end{definition}

\medskip

The following easy result is given in \cite[Lemma 2.2]{LuMr}:
\begin{proposition}
\label{Res-Ind}
If $\Gamma_0$ is a subgroup of index $m$ in $\Gamma$ then
\begin{align*}&R_n(\Gamma_0) \le m R_{mn} (\Gamma)\\
\mathrm{and} \; \; &R_n(\Gamma) \le m R_n(\Gamma_0)
\end{align*}
\end{proposition}
\begin{corollary}
$\rho(\Gamma_0) = \rho(\Gamma)$.
\label{Commensurable}
\end{corollary}
Now, if $\rho(\Gamma) < \infty$ then $\calz_\Gamma(s)$ indeed
defines a holomorphic function on the half plane $\{s \in\bbc
\mid Re\, s > \rho(\Gamma) \}$ and $r_n(\Gamma)$ is polynomially
bounded.

Let now $\Gamma = \G (\calo_S)$ as in Section 3.  Assume further
that $\Gamma$ has the CSP.  Then by Theorem~\ref{Lubotzky-Martin},
$\rho(\Gamma) < \infty$ and $\calz_\Gamma (s)$ is indeed a well
defined function on the half plane.  Moreover, let $\Gamma_0$ be a
finite index subgroup of $\Gamma$, as in \S 3, for which
$A(\Gamma_0) = \G (\bbc)^{\#S_\infty} \times \prod\limits_{v
\notin S} L_v$ with $L_v$ open in $\G (\calo_v)$ for every $v$ and
$L_v = \G (\calo_v)$ for almost every $v$.  (We can take $\Gamma_0
= \Gamma$ if $\ker (\widehat{\G (\calo_S)} \to \G (\hat\calo_S)) =
\{ e \})$. Since there is a one-to-one correspondence between
representations of $\Gamma $ and rational representations of
$A(\Gamma)$ and since every irreducible representations of a
product of groups decomposes in a unique way as a tensor product
of irreducible representations of the factor groups, we get an
``Euler factorization":

\begin{proposition}
\label{Euler}
\[ \calz_{\Gamma_0} (s) = \calz_{\G(\bbc)} (s)^{\#S_\infty} \cdot
\prod\limits_{v \notin S} \calz_{L_v} (s)\] where
$\calz_{\G(\bbc)} (s) \; (\mathrm{resp.} \; \; \calz_{L_v} (s)) $
is the representation zeta function counting the irreducible
rational (resp. continuous) representations of $\G(\bbc)$ (resp.
$L_v$).
\end{proposition}

Now if we look at $V(p) = \{ v \mid v \notin S, \,  v|p\}$ i.e. all
the valuations of $k$ (outside $S$)   which lie over a prime $p$,
then $\prod\limits_{v \in V(p)} \calz_{L_v} (s)$ will be called
the $p$-factor of $\calz_\Gamma (s) $ and it will be denoted
$\calz^p_\Gamma (s)$.  Similarly, $\calz_{\G(\bbc)} (s)^{\#
S_\infty}$ is the infinite (or archimedean) factor of the ``Euler
factorization".

It should be noted that unlike the classical Euler factorization,
$\calz^p_\Gamma(s)$ does not exactly encode the representations of
$p$-power dimension.
\begin{example} Let $\Gamma = \SL_3(\bbz)$, so
\[
A(\Gamma) = \SL_3(\bbc) \times  \prod\limits_p \SL_3 (\bbz_p)
\]
and $\calz_\Gamma(s) = \calz_{\SL_3(\bbc)} \times \prod\limits_p
\calz_{\SL_3 (\bbz_p)} (s)$. The degrees of the irreducible
representations  of the pro-finite group $\SL_3(\bbz_p)$ divide
its order (which is a super-natural number---see
\cite[\S1.4]{Ri}). As $\SL_3(\bbz_p)$ is a virtually pro-$p$ group
the set of these degrees is contained in a finite union of type
$\bigcup^{\ell(p)}_{j = 1} q_j(p) p^{\bbn}$.
\end{example}
The picture for the general case is similar.

In the next three sections we look more carefully at the local factors.

\section{The local factors of the zeta function: the factor at
infinity}

Let $\bg$ be a connected, simply connected, complex almost simple
algebraic group and $G=\bg(\bbc)$.  As before $\calz_G (s)$ is the
zeta function counting the rational representations of $G$.  For
example $\calz_{\SL_2(\bbc)}(s) = \zeta(s)$ the Riemann zeta
function since $\SL_2$ has a unique irreducible rational
representation of each degree.

In general, the irreducible representations of $G$ are
parametrized by their \emph{highest weights} as follows: Let $\Phi$ be
the root system of $\bg$ and $\varpi_1,\dots, \varpi_r$ the
fundamental weights.  Write $\bbn = \{ 0, 1, 2, \dots \}$, and
for each $(a_1, \dots, a_r) \in\bbn^r$
consider $\lambda = \Sigma a_i \varpi_i$.
The irreducible representations $V_\lambda$ are
parametrized by these weights $\lambda$. The Weyl dimension formula
gives:
\[
\dim \; V_\lambda = \prod\limits_{\a \in \Phi^+} \; \frac{\a^\vee
(\lambda + \rho)}{\a^\vee (\rho)}
\]
where $\Phi^+ $ is the set of positive roots, $\rho$ is half the
sum of the roots in $\Phi^+$, and $\a^\vee$ is the dual root to $\a
\in\Phi^+$.  Note that $\prod\limits_{\a \in \Phi^+} \;
\frac{1}{\a^\vee(\rho)}$ is a constant depending only on $\bg$ and
not on $\lambda$, while the numerator $\prod\limits_{\a \in\Phi^+}
\; \, \a^\vee (\lambda + \rho)$ is a product of $\kappa = | \Phi^+
|$ linear functions in $a_1, \dots, a_r$.

\begin{theorem}
\label{Archimedean}
The abscissa of convergence of $\calz_{\bg(\bbc)}
(s)$ is equal to $\frac r\kappa$, where $r = \rk \bg$ and
$\kappa = |\Phi^+|$ is the number of positive roots.
\end{theorem}
\begin{proof}
The description above implies that
$$\calz_G (s)
= \suml^\infty_{a_1 = 0} \cdots \suml^\infty_{a_r = 0}
(\dim V_{a_1\varpi_1 + \dots + a_r\varpi_r})^{-s}.$$

Thus we have a question of the following type:
Given an $r\times \kappa$ matrix $b_{ij}$ of non-negative integers
and a vector $c_j$ of positive integers, what is the abscissa
of convergence of the Dirichlet series
$$\sum_{a_1=0}^\infty\cdots\sum_{a_r=0}^\infty \Bigl\{\prod_{j=1}^\kappa
(b_{1j} a_1 + \cdots + b_{rj} a_r + c_j)\Bigr\}^{-s}.$$
If we focus attention on the cube
$$\{(a_1,\ldots,a_r)\mid 0\le a_1,\ldots,a_r < N\},$$
we see that a typical term in this part of the
sum is of size $O((N^\kappa)^{-s})$.  Since there are $N^r$ such terms,
one might guess that the abscissa of convergence corresponds to the
real value $s$ for which $(N^u)^{-s}$ is comparable to the reciprocal of
$N^r$, i.e. $s=r/\kappa$.  For generic choices of the matrix
$b_{ij}$, this turns out to be right.  On the other hand, there may be
subsets of the cube of substantial size for which the product
of the sums $b_{1j} a_1 + \cdots + b_{r,j}+c_j$ is much smaller than
$N^{\kappa}$.  This happens if $(a_1,\ldots,a_r)$ lies near many
of the hyperplanes $H_j: b_{1j}x_1+\cdots+b_{rj}x_j=0$.  (In our examples,
these $H_j$ are precisely the walls of the Weyl chambers.)

To see how this can work, consider the series
$$\sum_{a=0}^\infty\sum_{b=0}^\infty\sum_{c=0}^\infty
((a+1)(a+b+1)(a+2b+1)(c+1))^{-s}.$$
If we consider only the $N$ terms with $a=b=0$, we obtain the Riemann
zeta-function, which diverges at $s=1$, where our naive guess
gave convergence for $\Re(s)>3/4$.  The problem is that three of the four
rows of our matrix of coefficients lie in a two dimensional subspace.
In order to compute the abscissa of convergence in any particular case, we
need to examine both the generic behavior on cubes $[0,N-1]^r$
and also behavior near the $H_j$.  In fact, we may need to consider
cases in which the index is near several $H_j$ but much nearer to some
than to others.
In the proof below, all of this is handled by a combinatorial
strategy that breaks up $[0,N-1]^r$ into subsets according, roughly,
to an integer vector which approximates the vector of
logarithms of the distances of
an index $(a_1,\ldots,a_r)$ from each of the $H_j$.

We begin, though, with the easy direction, proving that
$\calz_G(s)$ diverges for  $s = \frac r\kappa$.
If for $\lambda = (a_1, \dots, a_r)$ and $m >0$,
we have $a_i \le m$ for every $i = 1,\dots, r$, then $\dim
V_\lambda \le c_0 m^\kappa$ for some absolute constant $c_0$
depending only on $\bg$ (since, as mentioned above, the numerator
of $\dim V_\lambda$ is a product of $\kappa$ linear functions of
the coefficients $a_i$).  Thus $(\dim V_\lambda)^{-r/\kappa} \ge
c_1 m^{-r} $ for some constant $c_1>0$. Look now at the partial
sums $S_j$ taken over all $\lambda = (a_1,\dots, a_r)$ with $2^j <
a_i \le 2^{j+1}$.  As there are $(2^{j+1} - 2^j)^r = 2^{jr}$
summands, and each of them contributes at least
$c_1(2^{j+1})^{-r}$, we have $S_j \ge c_1/{2^r}$.  The sets $S_j$
are disjoint so $\calz_G\left(\frac r \kappa\right) \ge
\suml^\infty_{j=1} c_1/2^r = \infty$.

We have now to prove that for every $s > \frac r \kappa$,
$\calz_G (s)$ converges. For each $j \in \bbn$, let
$\Psi_j(\lambda) $ denote
\[
\Phi \cap \Span_\bbr \{ \a \in \Phi\mid |\a^\vee(\lambda + \rho) | <
e^j\}
\]
It is not difficult to check that $\Psi_j(\lambda)$ is itself a
root system (reduced but not necessarily irreducible). Moreover,
we clearly have
\[
\Psi_1(\lambda) \subseteq \Psi_2(\lambda) \subseteq \dots
\]
and the sequence stabilizes at $\Phi$.

Now, if $\a \in \Psi_{j+1} (\lambda) \backslash \Psi_j(\lambda)$
then $\log |\a^\vee (\lambda + \rho)| = j + O (1)$ and so:
\begin{align}
\label{logdim}
\log\dim V_\lambda &= \suml_{\a \in\Phi^+} \log \a^\vee (\lambda
+ \rho) + O (1) = \\
&=\suml_{\a\in\Phi^+} \; \suml^\infty_{j = 1} \eta(\a, j) +
O(1)\nonumber \\
\mathrm{where \ } \; \eta(\a, j) &= \begin{cases} 1 &\a
\notin\Psi_j
(\lambda)\\
0 &\a\in\Psi_j(\lambda)\end{cases}\nonumber
\end{align}
The last sum is equal (up to a constant depending on ${\Phi}$ but
not on $\lambda$) to\break $\suml^\infty_{i = 1} (|\Phi^+| -
|\Psi_i(\lambda)^+|) + O(1)$.

Let us now evaluate $\calz_G (s)$ for $s = \frac r \kappa +
\epsilon$, for a fixed $\e > 0$: Every $\lambda $ gives rise to a
sequence of root subsystems

\begin{equation}
\label{systemchain} \Psi_1(\lambda) \subseteq \dots \subseteq
\Psi_\ell (\lambda) = \Phi.
\end{equation}
This is an increasing sequence but with possible repetitions. We
will sum on $\lambda$ (and hence on these sequences) according to
the subsequence which omits the repetitions.  So we sum over all
possible strictly increasing sequences of subsystems
\begin{equation}
\label{strict-chain}
\Phi_1 \subsetneq \Phi_2 \subsetneq \dots \subsetneq \Phi_k = \Phi.
\end{equation}
Note that $k \le r$
(since $\dim \Span\Phi = r$).
 A sequence of type
(\ref{systemchain}) determines (and is determined by) a sequence of type (\ref{strict-chain})
together with a sequence of positive integers $b_1 < b_2 < \dots <
b_k$, such that
\begin{equation}
\begin{cases}
&\Psi_1 (\lambda) =  \dots = \Psi_{b_1 -1} (\lambda ) = \emptyset,\\
&\Psi_{b_1} (\lambda) = \dots = \Psi_{b_2 -1} (\lambda) = \Phi_1,\\
&\Psi_{b_2} (\lambda) = \dots = \Psi_{b_3 -1} (\lambda) = \Phi_2,\\
&\dots\, .\end{cases}
\end{equation}

Choose now a basis $\{\a_1, \dots, \a_r \}$ for $\Phi$ such that
the first $c_1$ vectors span the space $\Span (\Phi_1)$, the first
$c_2$ span $\Span (\Phi_2)$ etc.  This implies that for some
constant $\delta_1 \ge 1$
\begin{equation}
\label{factorbounds} 0 < \a_i^\vee (\lambda+ \rho) \le \delta_1
e^{b_j}\quad \forall i\le c_j,\ i=1,2,\ldots,k.
\end{equation}
\medskip

 Now, given $\Phi_1 \subsetneq \dots \subsetneq \Phi_k$ we
will sum over all possible sequences $1 \le b_1 < b_2 < \dots <
b_k$. We claim next that the number of dominant weights giving
rise to a particular pair of sequences $\Phi_1 \subsetneq \dots
\subsetneq \Phi_k$ and $b_1 < b_2 < \dots < b_k$ is bounded above
by a constant $\delta_2$ times
\begin{equation}
\label{boundedweights}
\exp (b_1 \rk \Phi_1 + b_2 (\rk \Phi_2 - \rk \Phi_1)
+ \dots + b_k(\rk \Phi_k - \rk \Phi_{k-1}))
\end{equation}

To see this, observe that the map
\begin{equation}
D: \lambda \mapsto \left(\a^\vee_1 (\lambda),\dots,
\a^\vee_r(\lambda)\right)
\end{equation}
is an injective linear transformation from $\bbn^r$ (identified
with set of dominant weights via the map $(a_1,\dots, a_r)
\mapsto\lambda = \suml^r_{i=1} a_i\varpi_i$) to $\bbn^r$. The map
$$ \lambda \mapsto \left(\a^\vee_1(\lambda + \rho), \dots,
\a^\vee_r(\lambda + \rho)\right)$$ is therefore an injective
affine map.  We need to bound the size of the set of all $\lambda
\in\bbn^r$ which give rise to $\Phi_1 \subsetneq \dots \subsetneq
\Phi_k$ and $b_1 < \dots < b_k$.  Each such $\lambda$ satisfies
all the inequalities of (\ref{factorbounds}).  Since $\det D$ is a
constant, their number is indeed bounded by a constant $\delta_2$
times (\ref{boundedweights}).

Finally, for each $\lambda$ the contribution of $V_\lambda $ to
$\calz_G \left(\frac r \kappa + \e\right)$ is bounded above by
some constant $\delta_3$ times
\begin{equation}
\exp\left( - \left(\frac r\kappa + \e\right) \Bigl(b_1| \Phi^+_1| + b_2(|\Phi^+_2|
- |\Phi^+_1|) + \dots + b_k(|\Phi^+_k | - |\Phi_{k-1}^+|)\Bigr)\right).
\end{equation}
To see this, note that (\ref{logdim}) implies that  $\log \dim V_\lambda =
\suml^k_{i = 1} b_i(|\Phi^+_i| - |\Phi^+_{i-1}|) + O(1)$ where
$\Phi^+_0 = \emptyset$.

Thus, for a suitable constant $\delta_4> 0$,
\begin{align}
\calz_G &(\frac r\kappa +\e) \le \delta_4
\suml_{\emptyset = \Phi_0 \subset \Phi_1\subset \dots \subset \Phi_k
=
\Phi} \; \; \suml_{1\le b_1<b_2<\dots <b_k}\nonumber\\
&\exp\left(\suml^k_{i=1} b_i(\rk\Phi_i-\rk \Phi_{i-1})\right)
\exp\left(-(\frac r\kappa +\e)\suml^k_{i=1}
b_i(|\Phi_i^+ | - |\Phi^+_{i-1}|)\right)\\
=\delta_4\hskip -10pt\suml_{\emptyset=\Phi_0\subset
\dots\subset\Phi_k=\Phi} &\left(\suml_{1\le b_1<\dots<b_k}
\exp(\suml^k_{i=1} b_i\left[(\rk\Phi_i - \rk\Phi_{i-1})-(\frac
r\kappa +\e)(|\Phi^+_i| - |\Phi^+_{i-1}|)\right]\right) \nonumber
\end{align}

To evaluate this sum we will use the following elementary
convergence lemma:
\begin{lemma*}
For constants $a_1,\dots, a_k\in\bbr$ the series
$$\suml_{1\le b_1 < \dots < b_k} \exp (\suml^k_{i =1} a_i b_i)$$
converges if and only if
\begin{equation}
a_k < 0, \, a_{k-1} + a_k < 0, \, \dots, \,a_1 + \dots +a_k < 0
\end{equation}
\end{lemma*}

\begin{proof}
Let
$$F_k(a_1,\dots, a_k) \eqdef \sum_{1\le b_1 < \dots < b_k} \exp(a_1 b_1 + \dots +a_k b_k).$$
Then,
\begin{align*}
F_k(a_1,\dots, a_k) &=
\sum^\infty_{b_1=1} \exp (a_1 b_1)
\sum_{1 + b_1 \le b_2 < \dots < b_k} \exp (a_2 b_2 + \dots + a_k b_k) \\
 &=\sum^\infty_{b_1 = 1} \exp \left( (a_1 + \dots +a_k)
 b_1\right) \sum_{1 \le c_2 < \dots < c_k} \exp (a_2 c_2 + \dots
 + a_k c_k ) \\
 &=\, \frac{\exp(a_1 + \dots + a_k)}{1-\exp(a_1 + \dots + a_k)} \,
 F_{k-1} (a_2, \dots, a_k) \\
 &=\prod\limits^k_{n=1} \; \frac{\exp (a_n + \dots + a_k)}{1-\exp
 (a_n + \dots + a_k)}.
\end{align*}
and the lemma follows.
\end{proof}
\bigskip

\medskip

In our application $a_i = (\rk\Phi_i - \rk\Phi_{i - 1}) -
(\frac r\kappa +\e)(|\Phi^+_i| - |\Phi^+_{i-1}|)$ and hence for $i
= 1, \dots, k$,
\begin{equation}
\label{telescope}
a_k + a_{k-1} + \dots +a_i =
(\rk\Phi - \rk\Phi_{i-1}) - (\frac r\kappa + \e)(|\Phi^+| -
|\Phi^+_{i-1}|)
\end{equation}
 (where $\Phi_0 = \emptyset$).

We need to prove that (\ref{telescope}) is less than $0$ or equivalently:
\begin{equation}
\label{diff-quotient}
\frac{\rk\Phi-\rk\Phi_{i-1}}{|\Phi^+| - |\Phi^+_{i-1}|}\, < \, \frac
r\kappa +\e
\end{equation}

By inspection of all the pairs of irreducible root systems
$\Phi_{i-1} \subset \Phi$, one sees that
\begin{equation}
\label{stability}
\frac{\rk\Phi_{i-1}}{|\Phi^+_{i-1}|} \, \ge \,
\frac{\rk\Phi}{|\Phi^+|} \, = \,\frac r\kappa.
\end{equation}

One can also give a conceptual proof of this inequality based on
the observation that $\frac r\kappa = \frac 2 h$ where $h$ is the
Coxeter number.  Now, if $\Phi_{i - 1} \subset \Phi$, the Coxeter
numbers satisfy $h_{i-1} \le h$.  This can be seen, for example,
from the fact that the Coxeter number minus one is the largest
exponent of the irreducible root system and the fact that this is
non-decreasing for inclusions of root systems can be deduced by
comparing the orders of the corresponding Chevalley groups.

Now, if $\Phi_{i-1}$ is reducible, say, $\Phi_{i-1} = \Phi'_{i-1} \coprod \Phi^{''}_{i-1}$ then
\begin{equation}
\label{reducible-stability}
\frac{\rk\Phi_{i-1}}{|\Phi^+_{i-1}|} \, = \, \frac{\rk \Phi'_{i-1} +
\rk \Phi^{''}_{i-1}}{|\Phi^{'+}_{i-1} | + |\Phi^{''+}_{i-1}|} \,
\ge \, \min\left\{ \frac{\rk\Phi'_{i-1}}{|\Phi^+_{i-1}|}, \;
\frac{\rk\Phi^{''}_{i-1}}{|\Phi^{''}_{i-1}|}\right\}
\end{equation}

The last inequality of (\ref{reducible-stability}) follows from the fact that if $a, b,
c, d \in\bbn$ then $\frac{a+c}{b+d}\, \ge \, \min\{ \frac ab,
\frac cd\}$.  It now follows that (\ref{stability}) is true also if
$\Phi_{i-1}$ is not necessarily irreducible.

Another elementary property of $a, b, c, d \in \bbn$ is that if $a
\le c,\;  b \le d$ and $\frac ab \ge \frac cd$ then
$\frac{c-a}{d-b} \, \le \frac cd$.  Applying this for $a = \rk
\Phi_{i-1},\, b = | \Phi_{i-1}^+|, \, c = \rk\Phi$ and
$d = |\Phi^+|$ and using (\ref{stability}), we deduce that $\frac{\rk\Phi -
\rk\Phi_{i-1}}{|\Phi^+| - | \Phi^+_{i-1}|} \, \le \, \frac r\kappa$
and hence (\ref{diff-quotient}) holds.  This finishes the proof of Theorem~\ref{Archimedean}.
\end{proof}

\begin{remarks*}
\begin{enumerate}
\item
We prove Theorem~\ref{Archimedean} in the simply connected case because this is the
only case we need for the intended application, and because the parametrization of
irreducible representations is slightly simpler in this case than in general.
The theorem is true without this hypothesis, however, and the argument is
unchanged except that instead of summing over $\bbn^r$, we sum over its intersection
with some finite-index subgroup of $\bbz^n$.

\item As far as we know, the function $\zeta_G(s)$ first appeared in the literature
in a paper of Witten \cite{Wi},
which discussed its values at positive even integers.
If $\Sigma$ is a compact orientable surface of genus $g\ge 2$,
$G^c$ is a compact form of $G$, and
$E$ is a principal $G^c$-bundle over $\Sigma$,
$\zeta_G(s)$ arises in the computation of the volume of
the moduli space $M$ of flat connections on $E$ up to gauge
transformations.  More precisely, $M$ has a natural symplectic structure $\w$
and a natural volume form $\theta = \frac{\w^n}{n!}$, where $2n =
\dim M = (2g-2)\kappa$, $\kappa = |\Phi^+|$, and $\Phi$ is the
root system of $\bg$.  As $\w$ represents the first Chern class of
a natural line bundle over $M$, the volume of $M$ with respect to
$\theta$ is rational.

On the other hand, the same integral can be computed by means of a
decomposition of $\Sigma$ into $2g-2$ pairs of pants, and from this
description it can be shown
that up to a rational normalizing factor, $\intl_M\theta$ is
$${(2\pi)^{-\dim M}}\sum_\lambda (\dim V_\lambda)^{2-2g}.$$
This shows that $\calz_G (s) = \suml_\lambda (\dim V_\lambda)^{-s}$
has the ``zeta property" that its value at every
positive even integer is a rational number times a suitable (integral) power of $\pi$.

We have no reason to believe this property is shared by our ``global" zeta
functions of arithmetic groups, but neither can we disprove it.
\end{enumerate}
\end{remarks*}

\section{The $p$-local factor}

We shift our attention now to the local factors at the finite
primes, i.e., to $\calz_{L_v}(s)$ in the notation of Proposition~\ref{Euler}.
This is the representation zeta function of the group $L_v$
which is open in $\bg(\calo_v)$ and it is equal to $\bg (\calo_v)$
for almost all $v$.

When $\char (k) = 0, L_v$ has an open uniform pro-$p$ subgroup
(cf. \cite[Chapter~4]{DDMS}).  An important result of A. Jaikin-Zapirain
asserts:
\begin{theorem}[Jaikin-Zapirain \cite{Ja2}]
\label{Jaikin-rational}
Assume $\char(k) = 0$ and $p > 2$ or
if $p=2$, assume $L_v$ is uniform. Then $\calz_{L_v}(s)$ is a
rational function in $p^{-s}$. More precisely, there exist natural
numbers $k_1,\dots, k_t$ and functions $f_1 (p^{-s}), \dots, f_t
(p^{-s})$ rational (with rational coefficients) in $p^{-s}$ such
that
$$
\calz_{L_v}(s) = \suml^t_{i=1} k^{-s}_i f_i(p^{-s})
$$
\end{theorem}
\begin{problem}
Does a similar result hold when $\char (k) > 0$?
\end{problem}

Theorem~\ref{Jaikin-rational} is quite deep.  It is proved by using Howe's
interpretation of the Kirillov orbit method for uniform pro-$p$
groups [Ho].  This enabled Jaikin-Zapirain to present $\calz_{L_v} (s)$ as
a $p$-adic integral and then to appeal to the work of Denef [De]
on the rationality of such integrals.

Jaikin-Zapirain also made some explicit calculations.  His main example is:
\begin{theorem}
\label{SL2}
Let $\calo_v$ be the ring of integers of a local field.
Let $\calm$ be its maximal ideal,
$\bbf_q = \calo_v/\calm$ and $L_v = \SL_2(\calo_v)$.  If $q$ is odd, then
\begin{align*}
\calz_{L_v}&(s) = 1+q^{-s} +\frac{q-3}{2} (q+1)^{-s} +2
\left(\frac{q+1}{2}\right)^{-s} +\\
&+\frac{q-1}{2}\; \,  (q-1)^{-s} + 2\left(\frac{q-1}{2}\right)^{-s} +\\
+ &\frac{ 4q(\frac {q^2-1}{2})^{-s} + \frac{(q^2-1)}{2} \,
(q^2-q)^{-s} \, + \, \frac{(q-1)^2}{2} (q^2+q)^{-s}} {1-q^{-s+1}}
\end{align*}
\end{theorem}

The reader can immediately see that the function depends only on
$q$ and not on $\calo_v$ and the abscissa  of convergence of $L_v$
is always $\rho(L_v) = 1$ independently of $q$ and $\calo_v$ (see
also Proposition~\ref{compare-growth} below).
This is especially interesting since
for $\bg = \SL_2, \; r = \rk \bg = 1$ and $|\Phi^+| = 1$, so $\frac
r\kappa = 1$.

\bigskip

Now we consider the general situation.
Let $K$ be a non-archimedean local field and $\bg$ an absolutely
almost simple algebraic group over $K$.
Fix a $K$-embedding of $\bg$ in $\GL_n$ for some $n$, and let $U =
\bg(K)\cap \GL_n(\calo)$ where $\calo$ is the ring of integers of $K$.
We consider what can be said in general about $\rho(U)$.

Let $\pi$ be a uniformizer of $\calo$, $q = |\calo/\pi\calo|$, and
$$U_k = \ker \left(U\to \GL_n(\calo/\pi^k\calo)\right).$$

\begin{definition}
\begin{enumerate}
\item For a finite group $H$ we denote by $\og(H)$ the number of
its conjugacy  classes.

\item We define $\gamma(U)$ (which may, a priori, depend on the embedding of
$U$ in $\GL_n(\calo)$) as follows:
\begin{equation}
\gamma =\gamma (U) = \mathop{\lim\sup}\limits_{k\to\infty} \;
\frac{\log_q\og\left(U/U_k\right)}{k}
\end{equation}
\end{enumerate}
\end{definition}

\begin{proposition}
\label{Crude-bound}
Let $ \delta = \dim (\bg)$ then:
\[
\rho(U) \, \ge \, \frac{2\gamma}{\delta - \gamma}
\]
In other words: if $\mu = \frac \gamma\delta$ then $\rho(U) \ge
\frac{2\mu}{1-\mu}$.
\end{proposition}

\begin{proof} The quotient $U/U_k$ is of order approximately (up to
multiplicative constant) $q^{\delta k}$ and  has approximately
$q^{\gamma k}$ representations.  If $q^{ak}$ is the median value
of the degrees of these representations, then: $\half q^{\gamma k}
\cdot (q^{ak})^2\le q^{\delta k}$. Hence $\g + 2a\le \delta +
o(1)$ as $k \to \infty$ i.e. $a \le \frac{\delta - \gamma}{2}$.
This means that $U$ has at least $\half q^{\g k}$ irreducible
representations of degree at most $q^{\half (\delta -\gamma)k}$.
Hence $\rho(U) \ge \frac{2\g}{\delta - \g}$ as claimed.
\end{proof}

\begin{proposition}
\label{ssimp-bound}
Let $K$ be a non-archimedean local field,
$\bg$ a $K$-simple algebraic group and $U$ an open compact
subgroup of $\bg(K)$.  Then $\rho(U) \ge \frac r \kappa$ where $r$
is the absolute  rank of $\bg, \kappa = |\Phi^+|$, and $\Phi^+$ is
the set of positive roots in the absolute root system of $\bg$.
\end{proposition}
\begin{proof}
Fix an embedding of $\bg \hookrightarrow \GL_n$ and then
$U\left(\subset\bg(K) \subset \GL_n(K)\right)$ is commensurable
with $\bg(\calo)\eqadef \G(K) \cap \GL_n(\calo)$ where $\calo$ is
the ring of integers of $K$.  By Corollary~\ref{Commensurable}, we can replace
$U$ by $\bg(\calo)$. Let $U_k=\ker\left (U\to
\GL_r(\calo/\pi^k\calo)\right)$, where $\pi$ is a uniformizer of
$\calo$.  Now, $[U\colon U_k]$ is of order approximately (up
to a bounded multiplicative constant) $q^{k\dim \bg}$ where $q = [\calo: \pi
\calo]$.  Let $\T$ be a maximal torus of $\bg$.  Then $\T(K)\cap
U$ is a compact open subgroup of $\T(K)$ of dimension $r$.  Its
projection in $U/U_k$, denoted $\T(\calo/\pi^k\calo)$ is of
order approximately $q^{kr}$.  Fix a maximal torus in $\GL_n$ and let
$L$ be a finite extension of $K$ over which this torus splits; let $\calo_L$ denote
the ring of integers in $L$.  We can regard $U/U_k$ and its subgroup
$T(\calo/\pi^k\calo)$ as subgroups
of $\GL_n(\calo_L/\pi^k\calo_L)$, and as such, the latter group
can be conjugated into a diagonal subgroup.

We claim that for any local ring $(A,\m)$,
two diagonal elements of $\GL_n(A)$ are conjugate if and only
if their entries are the same up to order.  To prove this, we give
a basis-independent characterization of the multiplicity of an
``eigenvalue'' $\lambda\in A$
of a diagonalizable $A$-linear map $T$ from a
rank-$n$ free $A$-module to itself.  Namely, the multiplicity of $\lambda$
is the $(A/\m)$-dimension of the image of
$\ker(T-\lambda\mathrm{Id})\subset A^n$ in $(A/\m)^n$.
We remark that this property is not true for general commutative
rings.  For instance, if $e$ is an idempotent,

$$\begin{pmatrix}e&e-1\\1-e&e\end{pmatrix}
\begin{pmatrix}1&0\\0&0\end{pmatrix}
\begin{pmatrix}e&e-1\\1-e&e\end{pmatrix}^{-1} =
\begin{pmatrix}e&0\\0&1-e\end{pmatrix}.$$

As $\calo_L/\pi^k\calo_L$ is local,
it follows that an element $x \in
\T(\calo/\pi^k\calo)$ is conjugate to at most $n!$
elements of $\T(\calo/r^k\calo)$ within $\GL_n(\calo_L/\pi^k\calo_L)$ and therefore,
a forteriori, within $U/U_k$.  This shows
that $U/U_k$ has at least $cq^{kr}$ different conjugacy
classes, for some $c>0$ which does not depend on $k$, and hence this number of different representations.  By Proposition~\ref{Crude-bound}, $\rho(U) \ge
\frac{r}{\kappa}$ as claimed.
\end{proof}

\section{The $p$-local factor: anisotropic groups}

In this section we consider another class of examples for which
$\rho = \frac r\kappa$, namely the anisotropic groups
over local fields $K$ in characteristic zero. Let $D$ be
a division algebra over $K$ of
degree $d$ and $G' = \SL_1(D)$ the $K$-algebraic group of elements
of $D$ of norm one.  Thus $G'(K)$ is a compact virtually pro-$p$ group.
This is a $K$-form of $\bg = \SL_d$, i.e. over $\overline K$, the
algebraic group $\SL_1(D)$ is isomorphic to $\SL_d$.  Thus while
$\rk_K(G') = 0$, the absolute rank of $G'$ is $d-1$
and the absolute root system is that of $\SL_d$.  In particular,
$|\Phi^+| = \frac {d^2-d}{2}$ and $\frac r\kappa\,  = \, \frac
{d-1}{(d^2-d)/2} \, = \, \frac 2d$.

\begin{theorem}
\label{SL1}
If $\char (K) = 0$, then $\rho\left(G'(K)\right) = \frac 2d$.
\end{theorem}

\medskip
\begin{remark*} We cannot prove the result  in positive
characteristic but see Theorem~\ref{not-char-2} below.
\end{remark*}

Before starting the proof of Theorem~\ref{SL1}, let us give a ``linear
algebra" lemma to be used in the proof.
\begin{lemma}
\label{base-change}
Let $K\subset L$ be an extension of local fields with ring of
integers $\calo_K$ and $\calo_L$.  Let $\pi$ be a uniformizer of
$K$ and $r \in \bbn$.

\begin{enumerate}
\item If $T\colon \calo^n_K \to \calo^n_K$ is an injective
$\calo_K$-linear map, then
\begin{align*}
&|\ker T\otimes (\calo_L/\pi^r\calo_L)|
= |\coker T\otimes (\calo_L /\pi^r\calo_L) \\
=&|\ker T\otimes (\calo_K /\pi^r\calo_k)|^{[L:K]}
= |\coker T\otimes (\calo_K/\pi^r\calo_K)|^{[L\colon K]}
\end{align*}

\item If $U\colon \calo^n_L \to \calo^n_L$ is an injective
$\calo_L$-linear map and
\[
\Lambda = \{ x \in \calo^n_K | x \otimes 1 \in \image U \},
\]
then
\begin{align*}
&|\ker U \otimes (\calo_L / \pi^r\calo_L) | = | \coker U \otimes
(\calo_L/ \pi^r\calo_L) | \\
\le & | \ker ( \Lambda \otimes (\calo_K / \pi^r\calo_K) \to
(\calo_K / \pi^r\calo_K)^{n}|^{[L:K]}\\
= &|\coker (\Lambda \otimes (\calo_K/\pi^r\calo_K) \to
(\calo_K/\pi^r\calo_K)^n|^{[L:K]}
\end{align*}
\end{enumerate}
\end{lemma}

\begin{proof}
If $D$ and $C$ denote the kernel and cokernel of $T\otimes
(\calo_K/\pi^r\calo_K)$, respectively, then $|D|=|C|$ since the
two middle terms in
\[
0 \to D \to (\calo_K/\pi^r\calo_K)^n\to (\calo_K / \pi^r\calo_K)^n
\to C \to 0
\]
have equal order.

Now, as  $\calo_L$ is free of rank $[L:K]$ over $\calo_K$,
$\calo_L / \pi^r\calo_L$ is free over $\calo_K/\pi^r\calo_K$,
tensoring with it we obtained
\[
0 \to D \otimes (\calo_L/\pi^r\calo_L) \to (\calo_L
/\pi^r\calo_L)^n \to (\calo_L/\pi^r\calo_L)^n \to C\otimes
(\calo_L/\pi^r\calo_L)\to 0
\]
So,
\[
|D\otimes (\calo_L/\pi^r \calo_L)| = |D|^{[L:K]} = |C|^{[L:K]} =
|C\otimes (\calo_L/\pi^r\calo_L)|
\]
which gives (i).

For (ii), let $I$ denote the image of $U$.  Let $S$ and $T$ denote the
inclusion maps $I\hookrightarrow\calo_L^n$ and
$\Lambda\hookrightarrow \calo_K^n$.
As $T\otimes_{\calo_K}\calo_L$ factors through $S$,  the
image of $T\otimes_{\calo_K}(\calo_L/\pi^r\calo_L)$ is contained in
the image of $S\otimes_{\calo_L}(\calo_L/\pi^r\calo_L)$.  It follows that
$$|\coker T\otimes_{\calo_K}(\calo_L/\pi^r\calo_L)| \ge
|\coker S\otimes_{\calo_L}(\calo_L/\pi^r\calo_L)|.$$
We conclude by applying part (i) to $T$.
\end{proof}

\bigskip
\noindent\emph{Proof of Theorem~\ref{SL1}.} \ By Proposition~\ref{ssimp-bound},
$\rho\left( G'(K)\right)\ge \frac 2d$.  It suffices, therefore,
to prove only the upper bound.

Let us start by reviewing the ``orbit method" classifying the
representations of uniform pro-$p$ groups.  Recall that a
torsion-free pro-$p$ group $U$ is called \emph{uniform} if
$U^p\supseteq [U, U] \; \,(U^4\supset [U, U]$ if $p=2$). If $L$ is
the Lie $\bbz_p$-ring of $U$ (see \cite[\S8.2]{DDMS}) then the
irreducible representations of $U$ are in one-to-one
correspondence with the orbits of homomorphisms $\vp:(L, +) \to
\mu_{p^\infty}$ (where $\mu_{p^\infty}$ is the group of $p$-power
roots of unity).  By orbits here we mean orbits under the adjoint
action of $U$ on $L$. Given such $\vp$, with orbit $[\vp] $, then
the dimension of the corresponding representation is
$|[\vp]|^{1/2}$.  (For a detailed description see \cite{Ho} and
\cite{Ja2}.)

We can be more concrete in the setting of interest for the theorem.
Let $\calo$ be
the ring of integers of $K, \pi$ a uniformizer of $\calo$,
$D_0$ the maximal $\calo$-order of $D$ (which consists of all the
elements of $D$ whose reduced norm is in $\calo$), and $L$ the
subspace of all the elements of $D_0$ of reduced trace $0$.
The map $x \mapsto \exp (px)$ from $L$ to $D$ takes $L$ into a
uniform open subgroup of $\SL_1(D)$ which we will call $U$, whose
Lie ring is $L$.  Our goal is to prove that $\rho (U) \le \frac
r\kappa \,= \,\frac 2 d$.

By the orbit method described above, we have to classify the
characters of $L$. Let
$$L^* = \{ x \in D\mid \Trd_{D/K}(x) = 0\mbox{ and }
\Tr_{K/\bbq_p} \Trd_{D/K} (x L) \subseteq \bbz_p\}.$$

Given a pair $(x, k)$ with $x \in L^*$ and $k \in \bbn$, the map
\[L\to \bbz_p \to \bbz_p/p^k\bbz_p \to \mu_{p^k} \to
\mu_{p^\infty}\]

\[ y \mapsto m = \Tr_{K/\bbq_p} (\Trd_{D/K} (xy))\mapsto m \pmod{p^k}
\mapsto e^{\frac{2\pi i}{p^k} m}\] is a character of $L$.
As the Killing form is non-degenerate all characters are obtained
in this way.  There are, though, two types of repetition.

\bigskip

\begin{enumerate}
\item The pairs $(x, k)$ and $(px, k+1)$ induce the same character.
\item If $x_1 \equiv x_2 \pmod{p^k}$ in $L^*$ then
$(x_1, k)$ and $(x_2, k)$ induce the same character.

\vskip .1in
We should also consider a third kind of equivalence among pairs
$(x, k)$:
\vskip .1in

\item For every $u\in U$, $(x, k) \sim (x^u, k)$,
where $x^u$ is the image of $x$ under the conjugation
action of $u$.

\end{enumerate}

We will denote the equivalence class of $(x, k)$ by $[x, k]$.  So,
there is a one-to-one correspondence between the irreducible
representations of $U$ and the equivalence classes $[x, k]$.  The
representation space associated to $[x, k]$ will be denoted by
$V_{[x, k]}$.  Note that the equivalences (i) and (ii) preserve
character, while (iii) varies character within an equivalence class.
We will denote by $|[x, k]|$ the
number of characters associated with the equivalence class
$[x,k]$.  The orbit method implies that
\begin{equation}
\label{orbit-dimension}
\dim V_{[x, k]} = |[x,k]|^{1/2}.
\end{equation}

Note that $L^*$ contains $L$ as a subgroup of finite index,
so by equivalence (i)
we can always assume that $x \in L$. The size of the orbit
$|[x,k]|$ is equal to the index of the centralizer of $\exp(px)$
in $U/U(p^k)$
where $U(p^k)=\exp(p^{k+1}L)$.

The division algebra $D$ has finitely many maximal subfields $F_1,
\dots, F_c$ such that every $x \in D$ is conjugate to one of them
in $D$.  This implies by (iii) and by the fact that $U$ is mapped
onto a finite index subgroup of $D^*/Z(D^*)$, that we can sum up
only on $x \in F_i, i = 1, \dots, c$.  It therefore suffices to
treat the contribution of a single $F=F_i$.

The intersection $F\cap D_0$ is an order in $F$, and in
establishing an upper bound, we may count all $x \in \calo_F$ with
$\Tr_{F/K} (x) = 0$, where $\calo_F$ is the ring of integers of
$F$. Denote by $\pi$ a uniformizer of $\calo_F$.

Let $\iota_1, \dots, \iota_d$ denote the embeddings of $F$ into
$\overline\bbq_p$---a fixed algebraic closure of $\bbq_p$.
Let us denote by $\Phi$ the root system of type $A_{d-1}$ given
via the standard basis $\{e_1,\dots, e_d\}$ of $\bbr^d$ as
$$\Phi = \{e_s - e_t\mid 1 \le s,t \le d,\, s \neq t\}.$$
Given the pair $(x, k)$ as before, let
\[
\Psi_i(x, k) = \{ e_s - e_t \in \Phi\mid
p^k|\pi^i(\iota_s(x) - \iota_t(x))\}
\]
One can check that $\Psi_i(x, k)$ depends only on the equivalence
class $[x, k]$, so we will denote it $\Psi_i[x, k]$. This is an
increasing sequence
$$\Psi_1[x, k] \subseteq \dots \subseteq \Psi_\ell [x, k] = \Phi$$
of root subsystems of $\Phi$, where $\ell = ke$ and $e$ is
the ramification degree of $K$ over $\bbq_p$.

The formal similarity with (\ref{systemchain}) deserves explanation
or at least comment.  The key is surely to be found in comparing
the way in which the orbit method works in the two settings,
$p$-adic analytic and compact real Lie groups.  In the latter
case, we can view dominant weights as integral $W$-orbits in
$\mathfrak{t}^*$, where $\mathfrak{t}$ is a Cartan subalgebra
of the Lie algebra.  In each setting, the chain of root systems
$\Psi_i$ characterizes the distances of a linear functional
on a Cartan subalgebra to the walls of the Weyl chambers.  To say
more, we would need a unified way of viewing the Weyl dimension
formula and (\ref{orbit-dimension}).  We do not know of such a way,
but the similarities between the two proofs cry out for a unified
treatment.

\begin{claim1*}We have
\[
\dim V_{[x, k]} = \exp_q\left( \Bigl(\suml^\ell_{i = 1} | \Phi^+
\setminus \Psi_i[x, k]|\Bigr) - |\Phi^+|\right)\] where for $y \in \bbr,
\; \exp_q (y) = q^y$.
\end{claim1*}
\begin{proof}
$\dim V_{[x, k]} =$(Index of centralizer of  $x$ acting on
 $U/U(p^k))^{1/2} =$
\begin{align*}
&\mathop{=}\limits^{(1)} \qquad \qquad \exp_q
\half\left(\suml_{1\le s \neq t \le d} (\min\{i|e_s-e_t \in
\Psi_i[x, k]\}-1)\right)\\
&\mathop{=}\limits^{(2)} \qquad \qquad \exp_q\half \,\left(
(\suml^\ell_{i=1} |
\Phi\setminus\Psi_i[x, k]|)-|\Phi|\right)\\
&\mathop{=}\limits^{(3)} \qquad \qquad \exp_q \, \left(
(\suml^\ell_{i=1} |\Phi^+ \setminus \Psi^+_i [x, k]|) -
|\Phi^+|\right)
\end{align*}
Let us justify these equalities:   (2) just reverses the order of
summation between $i$ and $(s,t)$, while
(3) follows from the central symmetry of
all root systems involved.  Equality (1) needs two remarks: First, as
$x$-set: $U/U(p^k)$ is isomorphic to $L/p^k L$ via the map
$\lambda \in L\to \exp (p\lambda)$ which satisfies $\exp (x^{-1}
(p\lambda) x) = x^{-1} \exp(p\lambda) x$.

Secondly, the size of the centralizer of $x$ acting on $U/U(p^k)$
is therefore the same as the kernel of $ad(x) - \mathrm{Id}$ acting on
$L/p^kL$. Lemma~\ref{base-change} implies that this kernel can be computed in a
suitable extension which contains all the eigenvalues $\iota_s(x),
1 \le s \le d$. Then equality (1) becomes clear.
\end{proof}

We now turn to the question: how many representations (i.e.,
equivalence classes $[x, k]$) give rise to a specified sequence
\begin{equation*}  (\Psi_1 \le \cdots \le \Psi_\ell = \Phi).
\end{equation*}
The reader may note here that the structure of the proof is
similar to that of Theorem~\ref{Archimedean}.  There is, however,
a minor difference at this point.  In the former proof every increasing chain
actually occurs for some $\lambda$.
Here, typically, many chains will not occur at all.
This does not matter because at this point
we are interested only in upper bounds.

\begin{claim2*}
The number of equivalence classes $[x, k]$ giving rise to a
sequence $\Psi_1\le \dots \le \Psi_\ell = \Phi$ is bounded above
by $\exp_q\Big(\suml^\ell_{i = 1} \big((d-1) -
\rk\Psi_i\big)\Big)$
\end{claim2*}
\begin{proof}
By Lemma~\ref{base-change} we may assume that all the eigenvalues
of $x$ are in $\calo$.

We know that $\Psi_\ell [x, k] = \Phi$ and the question is for how
many $x$'s (counted $\mod p^k$),  $\; \, \Psi_{\ell - 1} [x, k] =
\Psi_{\ell - 1}$ etc.

For a fixed reduction of $x \pmod {\pi^{i-1}}$ we would like to
count the number of possibilities modulo $\pi^i$.  This means
counting vectors $\left(\iota_1(x), \dots, \iota_d(x)\right)$
subject to two additional conditions:

\begin{enumerate}

\item trace $= \suml^d_{j=1} \iota_j(x) = 0$
\item $\iota_s(x) \equiv \iota_t(x)\pmod {\pi^i}$ for all
$(s,t)$ such that $e_s-e_t \in \Psi_{\ell - i}$

\end{enumerate}

\bigskip

The dimension of the affine space of solutions to (ii) compatible
with the specified reduction of $x$ (mod $\pi^{i-1}$) equals $d-\rk
\Psi_{\ell - i}$ and condition (i) leads to $(d-1) - \rk
\Psi_{\ell-i}$. This proves claim 2.

\medskip

Theorem~\ref{SL1} follows now from claims 1 and 2 by a
computation identical to that carried out in the proof
of Theorem~\ref{Archimedean}.
\end{proof}

We can prove Theorem~\ref{SL1} only for characteristic $0$ where the
orbit method is available.  For quaternion algebras, however, we can
prove the analogous theorem in odd characteristic as well.

\begin{theorem}
\label{not-char-2}
Let $K$ be a non-archimedean local field not of characteristic
$2$, $D$ a central simple algebra over $K$ of degree $2$, $D_0$ a
maximal order of $K$, and $\Nrd\colon D^\times\to K^\times$ the
reduced norm.  Then
$$\rho(D_0^\times\cap\ker\Nrd)=1.$$
\end{theorem}
\begin{proof}
As in the proof of Theorem~\ref{SL1}, we only need
to establish the upper bound.
Lacking both a logarithm map and a satisfactory version of the
orbit method in general, we develop a crude substitute for each.

Let $\O$ be the ring of integers in $K$, $\pi$ a uniformizer, and
$q$ the order of the residue field of $\O$.  Let
$$U_i = 1+\pi^i \O\subset \O^\times,$$
$$H_i = 1+\pi^i D_0\subset D_0^\times,$$
and
$$G_i = \ker N_i,$$
where $N_i\colon H_i\to U_i$ is the restriction of $\Nrd$.
If $m < n\le 2m$, there are natural isomorphisms
$$e^{\O}_{m,n}\colon \O/\pi^{n-m}\O\to U_m/U_n$$
and
$$e^{D}_{m,n}\colon D_0/\pi^{n-m}D_0\to H_m/H_n$$
defined by
$$e_{m,n}^{\O}(a+\pi^{n-m}\O) = (1+a\pi^m)U_n$$
and
$$e_{m,n}^{D}(a+\pi^{n-m}D_0) = (1+a\pi^m)H_n.$$
In particular, $|U_m/U_{m+1}|=q$ and $|H_m/H_{m+1}|=q^4$ and
therefore
$$|U_m/U_n|=q^{n-m},\ |H_m/H_n|=q^{4(n-m)}.$$
The identity
$$\Nrd(1+a)= (1+a)(1+\bar a) = 1+\Trd(a)+\Nrd(a)$$
implies that the diagram
\begin{equation}
\label{e-iso}
\xymatrix{D_0/\pi^{n-m}D_0\ar[d]_{T_{n-m}}\ar[r]^{e^D_{m,n}}&H_m/H_n\ar[d]^{N_{m,n}}\\
\O/\pi^{n-m}\O\ar[r]^{e^{\O}_{m,n}}&U_m/U_n}
\end{equation}
commutes, where $T_{n-m}$ and $N_{m,n}$ denote the maps induced by
$\Trd$ and $\Nrd$ respectively.

Let $F\subset D$ be a separable quadratic extension of $K$. Then
$$\Trd(F)=\mathrm{Tr}_{F/K}(F)=K,$$
and $D_0\cap F$ is an open subring of $F$, so $\Trd(D_0)\supset\Trd(D_0\cap F)$
contains an open subgroup of $\O$. If $\pi^{c_1}\O\subset
\Trd(D_0)$, then $\pi^{r+c_1}\O\subset \Trd(\pi^r D_0)$ for all
non-negative integers $r$.  This implies that if $r>c_1$,
$$(\Nrd(1+\pi^r D_0)\cap U_{r+c_1})U_{r+c_1+1} = U_{r+c_1}$$
and therefore by the completeness of $K$ that
\begin{equation}
\label{open-norm}
\Nrd(H_r)\supset U_{r+c_1}
\end{equation}
for all $r>c_1$.  This implies $|\coker N_r| < c_2$ for some
constant $c_2$ independent of $r$.

Let $L$ denote the
Lie ring of elements of reduced trace $0$ in $D_0$.
Applying the snake lemma to
\begin{equation*}
\xymatrix{0\ar[r]&\pi^{n-m}D_0\ar[r]\ar[d]&D_0\ar[r]\ar[d]
    &D_0/\pi^{n-m}D_0\ar[r]\ar[d]^{T_{n-m}}&0\\
0\ar[r]&\pi^{n-m}\O\ar[r]&\O\ar[r]&\O/\pi^{n-m}\O\ar[r]&0\rlap{,}}
\end{equation*}
we see that
$$\frac{|\ker T_{n-m}|}{|L/\pi^{n-m}L|} \le c_3$$
and
$$|\coker T_{n-m}|\le c_3$$
for some constant $c_3$.  When $m\le n\le 2m$, we therefore obtain
that
$$|\coker N_{m,n}| \le c_3.$$

This inequality allows us to estimate $|G_1/G_n|$. Thus,
$$q^{3(n-m)}\le |\ker N_{m,n}|=\frac{|H_m/H_n||\coker N_{m,n}|}{|U_m/U_n|}
\le c_3q^{3(n-m)}.$$
The snake lemma gives
$$1\to G_m/G_n\to\ker N_{m,n}\to \coker N_n$$
and therefore
$$\frac{|\ker N_{m,n}|}{|\coker N_n|}\le |G_m/G_n|
    = |\ker N_m/\ker N_n|\le |\ker N_{m,n}|\le c_3q^{3(n-m)}.$$
Applying this for $m=\lceil n/2\rceil$ and iterating, we get
\begin{equation}
\label{size} \log |G_1/G_m| = 3m\log q + O(\log m).
\end{equation}
Next we directly compare $G_m/G_n$ and $L/\pi^{n-m}L$ for
$c_1\le m < n\le 2m$. The diagram (\ref{e-iso}) determines an
$H_1$-equivariant isomorphism
$$\ker T_{n-m}\tilde\to \ker N_{m,n}.$$
We know that $L/\pi^{n-m}L$ and $G_m/G_n$ can each be realized
as a subgroup of index $\le c_3$ in $\ker N_{m,n}$.  There is a natural
relation between characters on the two subgroups according to
which a character on the first group corresponds to a character on
the second if each is the restriction of a common character on
$\ker N_{m,n}$. Each character on either subgroup extends to at least
$1$ and at most $c_3$ characters on $\ker N_{m,n}$. Each subgroup has a
natural filtration, one arising from the filtration of $G_m$ by
$G_{m+i}$ and the other from the filtration of $L$ by $\pi^i
L$.  These can be compared.  For our purposes it is enough to
note that the image of $G_{n-1}/G_n$ in $\ker N_{m,n}$ is
contained in the image of $\pi^{n-m-1}L/\pi^{n-m}L$ in $\ker
T_{n-m}$. It suffices to check that for every $a\in L$, we have
$\Nrd(1+\pi^{n-1}a)\in \Nrd(H_n)$.  This follows immediately from
(\ref{open-norm}) since
$\Nrd(1+\pi^{n-1}a) \in U_{2n-2}.$

Every continuous irreducible complex representation of $G_1$ is a
representation of $G_1/G_n$ for some minimal $n$ which we call the
\emph{level} of the representation.  We would like lower bounds
for the dimensions of representations $V$ of level $n\ge 2c_1$.
Let $m=\lceil n/2\rceil$, and consider the restriction of $V$ to
the normal abelian subgroup $G_m/G_n$ of $G_1/G_n$.  As
$V^{G_{n-1}/G_n}$ is a $G_1/G_n$-subrepresentation of $V$, it must
be trivial, since $V$ is of level $n$. Let $\chi_0$ denote a
character of $G_m/G_n$ appearing in the restriction of $V$ to
$G_m/G_n$.  Every character in the $G_1/G_n$-orbit of $\chi_0$ in
$\Hom(G_m/G_n,\C^\times)$ appears in this restriction, so the
dimension of $V$ is at least as large as the orbit of $\chi_0$,
which is a character non-trivial on $G_{n-1}/G_n$. To find a lower
bound for the size of this set, we need an upper bound on the size
of the stabilizer of $\chi_0$.

Let $\phi_0$ denote a character of $L/\pi^{n-m}L$ which
corresponds to $\chi_0$.  The orbit size of $\phi_0$ differs from
that of $\chi_0$ by at most a factor of $c_3$. As $\chi_0$ is
non-trivial on $G_{n-1}/G_n$, $\phi_0$ is non-trivial on
$\pi^{n-m-1}L/\pi^{n-m}L$. We therefore proceed by setting
$r=n-m$ and finding a lower bound for the size of the orbit of
$$\phi_0\in\Hom(L/\pi^r L,\C^\times)\setminus
    \Hom(L/\pi^{r-1} L,\C^\times).$$

To get this, we fix a character $\chi\colon K/\O\to\C^\times$ such
that $\chi(\pi^{-1})\neq 1$.  This gives a pairing
$$\langle a,b\rangle = \chi(\Trd(ab))$$
on $L\otimes K$.  Let $L^*$ denote the kernel of $L$,
i.e., the set of $b$ such that $\langle a,b\rangle = 1$ for all
$a\in L$. Thus $L^*$ is an $\O$-lattice in $L\otimes K$,
and $L\subset L^*$.  The map
$$a\mapsto \langle \pi^{-r}a,x\rangle$$
gives an isomorphism
$$L^*/\pi^r L^*\to \Hom(L/\pi^r L,\C^\times).$$
As $L$ and $L^*$ are commensurable, the minimum orbit size
of an element of $L^*/\pi^r L^*$ not divisible by
$\pi$ and an element of $L/\pi^r L$ not divisible by $\pi$
differ by a bounded factor.

We are therefore led to the question of estimating the size of the
stabilizer in $G_1/G_r$ of $x_0\in L/\pi^r L$ under the
hypothesis $\pi\nmid x_0$. We begin by analyzing the set
$$C_{x,r}=\{y\in D_0\mid xy-yx\in \pi^r D_0\}.$$
for a fixed $x$.  We see that the set $S$ of pairs
\begin{align*}
&\{(x,y)\mid x\in L\setminus \pi L,\ y\in D_0,\ y\notin
    \Span_K\{1,x\}+\pi D_0\} \\
=&\{(x,y)\mid x\in L\setminus \pi L,\ y\in D_0,\ y\notin
    \Span_{\O}\{1,x\}+\pi D_0\} \\
\end{align*}
is compact and contains no pair of commuting elements.  This is
because in characteristic $\neq 2$, any trace-zero element of
$M_2(\bar K)\supset D\supset D_0$ commutes only with linear
combinations of itself and the identity. By compactness, there
exists $c_4$ such that $xy-yx\notin \pi^{c_4}D_0$ for all $(x,y)\in
S$.  It follows that every $y\in B_{x,r}$ is congruent (mod
$\pi^{r-c_4}$) to an element of $\Span_K\{1,x\}$ and therefore
congruent (mod $\pi^{r-c_4}$) to an element of $\Span_{\O}\{1,x\}$.
We count the number of elements of $y\in B_{x_0,r}$ such that
$\Nrd(y)=1$ up to congruence (mod $\pi^{r-c_4}$).  Without loss of
generality, we may assume that $y=u+vx_0$, where $u,v\in \O$.  The
norm condition asserts
$$\Nrd(u+vx_0) = u^2 + \Nrd(x_0) v^2 = 1,$$
so we count the number of solutions $(u,v)$ of this equation in
the ring $\O/\pi^{r-c_4}\O$.  The solution set is the union of
solutions where $\pi\nmid u$ and solutions where $\pi\nmid v$. If
$(u_1,v)$ and $(u_2,v)$ are solutions of the first kind, then
$(u_1+u_2)(u_1-u_2) = 0$, and $(u_1+u_2)+(u_1-u_2) = 2u_1$. As $u_1$
is a unit, this implies that the g.c.d. of $u_1+u_2$ and
$u_1-u_2$ divides $2$.  If $2$
is $\pi^{c_5}$ times a unit, then
$$u_1\equiv \pm u_2 \pmod{\pi^{r-c_4-c_5}}.$$
This gives at most $2q^{c_5}q^{r-c_4}$ solutions where $\pi\nmid
u$. The same argument applies when $\pi\nmid v$ as long as
$\pi\mid \Nrd(x_0)$. If $\pi\nmid\Nrd(x_0)$, the same argument
applies with the roles of $u$ and $v$ exchanged.

We conclude that
$$|\Stab_{G_1/G_r}(x_0| \le q^{2c_4}(4q^{r+c_5-c_4}) = c_6 q^r.$$
From this bound and (\ref{size}), we can estimate the size of the
orbit $O(x_0)$:
$$\log|O(x_0)| = 2r\log q + O(\log r).$$
Thus, for all representations $V$ of level $n$,
\begin{equation}
\label{level-size} \log \dim V \ge 2r\log q+O(\log r).
\end{equation}
As $r = n-m = [n/2]$,
$$|G_1/G_n| = 6r\log q+O(\log r).$$
Since the sum of the squares of dimensions of all the
representations of $G_1/G_n$ is $|G_1/G_n|$, the number of
representations of level $n$ satisfies
\begin{equation}
\label{level-number} \log|\{V\mid \textrm{level}(V)=n\}| \le
2r\log q + O(\log r).
\end{equation}
Together, (\ref{level-size}) and (\ref{level-number}) imply the
theorem.
\end{proof}

\section{The $p$-local factor: isotropic groups}

Theorems~\ref{Archimedean},\ref{SL1}, \ref{not-char-2}, and \ref{SL2}
and Proposition~\ref{ssimp-bound}
strongly suggest that
$\frac r\kappa$ is always the abscissa of convergence.
Our work on the subject was dominated for quite a long time by an
effort to prove this.  It turns out though that all these examples
were misleading and in fact we have:
\begin{theorem}
\label{Isotropic}
If $K$ is a local
non-archimedean field, $\bg$ an \emph{isotropic}  simple $K$-group
(i.e., $\rk_K \G\ge 1$) and $U$ an open compact
subgroup of $\bg(K)$, then $\rho(U) \ge \frac1{15}$.
\end{theorem}

\bigskip

\begin{remarks*}
\begin{enumerate}
\item If $r = \rk(\G)$
goes to infinity then $\frac r \kappa \to 0$, thus Theorem~\ref{Isotropic}
shows that $\frac r \kappa$ is usually not the abscissa of
convergence. It still may be the right answer for groups of low $K$-rank.

\item The difference between isotropic and anisotropic groups is
expressed by the fact that the first have non-trivial Bruhat-Tits
building.  The proof we give below does not refer to the
buildings, but it seems that a better combinatorial understanding
of them may lead to a more precise estimate on $\rho(U)$ (see \S
11 for more).

\end{enumerate}
\end{remarks*}
\begin{proof}

It suffices to treat the case of absolutely almost simply groups
over $K$. Moreover, it suffices to treat one representative from
every isogeny class. Tits  \cite{Ti} gives  a full description of
the classical absolutely simple groups  over  $K$. Note that in
proving our theorem we may ignore the groups of type $G_2$, $F_4$,
$E_6$, $E_7$ and $E_8$ as for these groups $\frac r\kappa \ge
\frac{1}{15}$, so our theorem follows from
Proposition~\ref{ssimp-bound}. Similarly, we can ignore forms of
${}^3D_4 $ and ${}^6D_4$.  All the rest are given in \cite[Table
II]{Ti} up to isogeny as groups of one of the following classical
forms:
\renewcommand{\theenumi}{\alph{enumi}}
\begin{enumerate}

\item $\SL_m(D)$ where $D$ is a central division algebra over $K$ of
degree $d$. These are inner forms of $A_n$ for $n = m d -1$ (and
we can assume $m \ge 2$ as we consider only isotropic groups).

\item $\SU_m(L, f)$ where $L$ is a quadratic extension of $K$ and $f$ is
a non-degenerate hermitian form of index $\ix\ge m/2-1$.  These are outer forms of
$A_{m-1}$.

\item $\SO_m(K,f)$, where $f$ is a quadratic form of index $\ix\ge m/2-2$.  These are inner
forms of $B_{(m-1)/2}$ if $m$ is odd, and they are forms
of  $D_{m/2}$ if $m$ is even, outer or inner according to whether $m/2-\ix$ is odd or even.

\item $\Sp_m(K)$, where $m$ is even.  These are the groups of type
$C_{m/2}$ and have index $\ix=m/2$.

\item $\SU_m(D,f)$, where $D$ is the quaternion algebra over $K$ and $f$ is a non-degenerate
antihermitian form of index $\ix \ge (m-1)/2$.  These are inner forms of $C_m$.

\item $\SU_m(D,f)$, where $D$ is the quaternion algebra over $K$ and $f$ is a non-degenerate
hermitian form of index $\ix \ge (m-3)/2$.  These are forms of $D_m$, outer or inner depending
on $m - 2\ix$.

\end{enumerate}
\renewcommand{\theenumi}{\roman{enumi}}

\medskip

We will start with case (a). For simplicity
we will assume first that $D = K$ and $m$ is even.  We
later remark how to modify the proof for the general case.

Let $X$ and $ Z$ be diagonal $\frac m2 \times \frac m2$ matrices
such that $\binom{X \; 0}{0 \; Z}$ is a diagonal matrix in
$\SL_m(\calo)$ which is regular and has trace 0. For some $t$, all
the diagonal entries are distinct (mod $\pi^t$).

We will consider the matrices $M_Y$
obtained by reducing $I + \pi^k \binom{X \; 0}{0 \; Z} + \binom{0 \; Y}{0 \; 0}$
modulo $\pi^{3k+2t}$,
where $X$ and $Z$  are fixed and $Y$
varies over the $q^{\frac {m^2}{4} (3k+2t)}$ possibilities (mod $\pi^{3k+2t}$).
Such matrices have determinant $1$ (mod $\pi^{2k}$), so assuming that
$k > t$, without sacrificing (mod $\pi^t$) regularity,
we can always modify $Z$ (mod $\pi^k$) to arrange that
$M_Y\in \SL_m(\calo/\pi^{3k+2t}\calo)$ for all $Y$

Assume two such matrices $M_Y $ and $M_{Y'}$ are conjugate.  This
means that there is an $m\times m$ matrix $\binom{A \; B}{C \; D}$,
where $A, B, C, D \in M_{\frac m2} (\calo/\pi^{3k+2t}\calo)$,
$\det{\binom{A \; B}{C \; D}} = 1$ and
\begin{align}
\label{block-conj}
\begin{pmatrix}A&B\\C&D\end{pmatrix}\; \;
\begin{pmatrix}\pi^kX&Y\\0&\pi^kZ\end{pmatrix}\,
= \, \begin{pmatrix}\pi^kX&Y'\\ 0&\pi^kZ\end{pmatrix}\;
\begin{pmatrix}A&B\\ C&D\end{pmatrix}
\end{align}

From $(\ref{block-conj})$ we can deduce:

\bigskip

\begin{enumerate}
\item Considering the lower left block,
$$CX \equiv ZC \pmod {\pi^{2k+2t}},$$
which implies $C\equiv 0 \pmod {\pi^{2k+t}}$ since
the difference of any diagonal entry of $X$ and any diagonal entry of
$Z$ cannot be divisible by $\pi^{t+1}$.

\item Considering the upper left block,
$$\pi^kAX \equiv \pi^k X A + Y'C \pmod {\pi^{3k+2t}}$$
From (i) we know that $C \equiv 0 \pmod{\pi^{2k+t}}$, so we get
$$AX \equiv XA \pmod{\pi^{k+t}}$$
which implies that $A$ is diagonal (mod $\pi^k$) since
the difference between two distinct diagonal entries of
$A$ cannot be divisible by $\pi^{t+1}$.

\item Considering the lower right
block,
$$CY + \pi^kDZ\equiv \pi^k ZD\pmod{\pi^{3k}},$$
and hence by (i),
$$DZ \equiv ZD \pmod{\pi^{k+t}}$$
and so $D$ is diagonal (mod $\pi^k$).

\item From (ii) and (iii) and $\det \begin{pmatrix}A&B\\C&D\end{pmatrix}=1$, we deduce that $A$ and $D$
are invertible (mod $\pi$) (since $C\equiv 0 \pmod{\pi}$) and hence
also (mod $\pi^k$).

\bigskip

\item The upper right corner now gives
\begin{equation*} AY \equiv Y' D \pmod{\pi^k}.
\end{equation*}
So,
\begin{equation*} AYD^{-1} \equiv Y' \pmod{\pi^k}.
\end{equation*}

\end{enumerate}

Let now
$$\tilde M = \{ M_{Y}\mid Y  \in M_{\frac m2} (\calo/\pi^{3k+2t}\calo)\}.$$
Choose out of $\tilde M$ a set $M$ of
$q^{\frac{m^2}{4}k}$ representatives for the different possible
images of $Y$ (mod $\pi^k$).

Assertions (ii)--(v) imply  that for a given such $M_Y\in M$, there are at most $ q^{(2\cdot \frac {m}{2} -1)k}
 = q^{(m-1)k}$ other elements of $M$ which are in the same
 conjugacy class.  This implies that $M$ has representatives of at
 least $q^{(\frac{m^2}{4} - m+1)k}$ different conjugacy classes.
 In particular, $\SL_m(\calo /\pi^{3k+2t} \calo)$ has at least
 $q^{(\frac{m^2}{4}-m+1)k}$ conjugacy classes.
 So we have proved that
 \[
 \g\left(\SL_m(\calo)\right)\ge \frac13(\frac{m^2}{4} - m + 1) =
 \frac{1}{12} (m^2-4m+4)
 \]

 Proposition~\ref{Crude-bound} now implies that
 \[
 \rho\left(\SL_m(\calo)\right) \ge \frac{\frac16
 (m^2-4m+4)}{m^2-1-\frac{1}{12} (m^2-4m+4)}
 \]

This proves the theorem for every $m\ge 6$ even. For $m=4$ we can
use Proposition~\ref{ssimp-bound}.

\medskip

If we replace $K$ by $D$ (and $m$ is still even) the proof works
as it stands (recall that $\SL_m(D)$ means the set of all $m \times m$
matrices over $D$ whose determinant, considered as an $md\times
md$ matrix over $\bar K$, is one.)  The only modification needed is
when computing dimensions:  the number of elements of $M$ is
$q^{\frac{m^2d^2}{4} k}$ and every element there can be conjugated
to at most $q^{(md-1)k}$ other elements of $M$ (since
$\rk(\SL_m(D)) = md -1$.)

We deduce that for an open subgroup $U$ of $\SL_m(D)$, $\g (U) \ge
\frac 13 (\frac 14 m^2d^2-md+1)$ and since $\dim_k(\SL_m(D)) =
m^2d^2-1$ we have
\[
\rho(U)\, \ge\, \frac{\frac 23 (\frac 14
m^2d^2-md+1)}{m^2d^2-1-\frac 13(\frac 14 m^2d^2-md+1)}
\]

This proves the theorem for $md\ge 6$.  For $md\le 4$, we can use
Proposition~\ref{ssimp-bound}.
\medskip

Finally, for general $m$ we will write $m$ as
$m=\left[\frac{m+1}{2}\right] \, + \, \left[\frac m2\right]$ and
in the proof we will work with blocks of sizes
$\left[\frac{m+1}{2}\right]$ and $\left[\frac m2\right]$. E.g.,
$X$ is an $\left[\frac{m+1}{2}\right]\times
\left[\frac{m+1}{2}\right]$ matrix, $Y$ is
$\left[\frac{m+1}{2}\right]\times \left[\frac m2\right], Z$ is
$\left[\frac m2\right]\times \left[\frac m2\right]$ etc. The proof
(for $K$ or $D$) carries over without any difficulty.  The size of
$M$ is then
$q^{\left[\frac{m+1}{2}\right] \, \left[\frac m2\right] d^2k}$,
and every element of it is conjugate to at most
$q^{(md-1)k}$ elements,  so
$$\g(U)\ge \frac 13
(\left[\frac{m+1}{2}\right]\left[\frac m2\right] d^2-md+1),$$
and
$$\rho(U) \ge \frac{\frac 23(\left[\frac{m+1}{2}\right]\left[\frac m2 \right]
d^2-md+1)}{m^2d^2-1-\frac 13(\frac 14m^2d^2-md+1)}.$$

This time we can assume $m\ge 3$.  For $m\ge 3$ odd, the only pairs $(m,d)$
for which this quantity is less than
$\frac 1{15}$ are $(3,1)$, $(3,2)$, and $(5,1)$.
For these cases we can use
Proposition~\ref{ssimp-bound}.

\bigskip

We now turn to groups of type (b)--(f), i.e. $\SU_m(D, f)$ where $D$ is
either $K, L$ or the quaternion algebra over $K$ and $f$
is a Hermitian of skew-Hermitian form on $W=D^m$.  By definition of index,
we can choose a basis
$$e_1,\ldots,e_{\ix},f_1,\ldots,f_{\ix},g_1,\ldots,g_s$$
with respect to which our form can be written in blocks of sizes
$\ix$, $\ix$, and $m-2\ix$ as follows:
$$
\begin{pmatrix}
0&I&0\\
\pm I&0&0\\
0&0&N
\end{pmatrix}
$$
(Note that if $m=2\ix$, the third block size is zero, so in fact we will have a $2\times 2$
block matrix.)

For fixed $X$ and $Z, x \times x$ matrices, we consider matrices
of the form
$$M_Y =  I_m +
\begin{pmatrix}
\pi^k X&Y&0\\
0&\pi^k Z&0\\
0&0&0
\end{pmatrix}.
$$
When $X=Z=0$, the condition on $Y$ for this matrix to lie in $\SU_m(D,f)$ is
$Y \pm \sigma(Y) = 0$, where $\sigma$ is
the involution (possibly trivial) defining the group.  The number of distinct possibilities
for $Y$ (mod $\pi^k$) is $q^{\ix^2 k}$, $q^{\frac{\ix^2-\ix}2 k}$, $q^{\frac{\ix^2+\ix}2 k}$,
$q^{(2\ix^2+\ix)k}$, and $q^{(2\ix^2-\ix)k}$ for cases (b), (c), (d), (e), and (f)
respectively.
By conjugation, we see that whenever $X$ and $Z$ are chosen so that
$M_0\in \SU_m(D,f)$, the number of possible values of $Y$ (mod $\pi^k$) for
which $M_Y\in\SU_m(D,f)$ is the same.
We will fix $X, Z$ to be diagonal matrices which have all $2\ix$ entries
distinct (mod $\pi^t$).  For some value of $t$
depending only on $m$ and the order $q$ of the residue field of $K$, we can always do this.

If $M_Y$ and $M_{Y'}$ are conjugate, we have a (mod $\pi^{3k+2t}$) equality:
$$
\begin{pmatrix}
A&B&C\\
D&E&F\\
G&H&J\\
\end{pmatrix}
\begin{pmatrix}
I_{\ix}+X&Y&0\\
0&I_{\ix}+Z&0\\
0&0&I_s\\
\end{pmatrix}
=
\begin{pmatrix}
I_{\ix}+X&Y'&0\\
0&I_{\ix}+Z&0\\
0&0&I_s\\
\end{pmatrix}
\begin{pmatrix}
A&B&C\\
D&E&F\\
G&H&J\\
\end{pmatrix}.
$$

Imitating the steps (i)--(v) above, we prove first that $D\equiv 0$ (mod $\pi^{2k+t}$) and
next that $A$ and $E$ are diagonal (mod $\pi^k$).  Finally, we conclude that
there are at least $q^n$ distinct conjugacy classes, where $n$ is
$\ix^2 k - 2\ix k$, $\frac{\ix^2-\ix}2 k - \ix k$, $\frac{\ix^2+\ix}2 k - \ix k$,
$(2\ix^2+\ix)k - 4\ix k$, and $(2\ix^2-\ix)k - 4\ix k$ for cases (b) through (f) respectively.  Using
the relation between $m$ and $\ix$, we conclude that
$$\rho(U)\ge
\begin{cases}
\frac{2\ix^2 - 4\ix}{(2\ix+2)^2-1- (\ix^2 - 2\ix)}&\text{in case (b)}\\
\frac{\ix^2-3\ix}{(2\ix+4)(2\ix+3)/2 - (\ix^2-3\ix)/2}&\text{in case (c)}\\
\frac{\ix^2-\ix}{2\ix(2\ix+1)/2-(\ix^2-\ix)/2}&\text{in case (d)}\\
\frac{4\ix^2+2\ix}{(4\ix+2)(4\ix+1)/2 - (2\ix^2+\ix)}&\text{in case (e)}\\
\frac{4\ix^2-2\ix}{(4\ix+6)(4\ix+5)/2 - (2\ix^2-\ix)}&\text{in case (f)}\\
\end{cases}
$$
In all cases, therefore, $\rho(U)\ge \frac{2\ix^2-6\ix}{3\ix^2+17\ix+12}$, so $\rho(U)>\frac 1{15}$
for $\ix\ge 5$.  For $\ix\le 4$, we have rank $\le 11$, and therefore
Coxeter number $\le 30$.  Thus, Proposition~\ref{ssimp-bound} covers
all these cases.

\end{proof}

\medskip

\begin{remarks*} Theorems~\ref{SL1} and \ref{Isotropic} show a dichotomy in the
asymptotic behavior of $\rho(U)$ between isotropic and anisotropic
groups. It should be noted, however, that the number $\rho(U)$
itself cannot distinguish between the two: For example, for a
quaternion algebra $D$ over $K$ of characteristic  zero $\rho
(\SL_1(D)) = 1$ and at the same time  $\rho(\SL_2(\calo)) = 1$ when
$\calo$ is the ring of integers of $K$.
\end{remarks*}

\section{Applications to general groups}

We can now apply the results of the previous section to prove the
following theorem:
\begin{theorem}
\label{Linear-bound} If $\Gamma $ is a finitely generated infinite
linear group over some field $F$ (or more generally, if $\Gamma$
is a finitely generated group with some homomorphism $\vp:\Gamma
\to \GL_n(F)$ with $\vp(\Gamma)$
 infinite) then $\rho(\Gamma) \ge \frac 1{15}$.

\end{theorem}

\begin{proof}  If $\Gamma$ is a quotient of $\Gamma_1$ then
$0 \le \rho (\Gamma) \le \rho(\Gamma_1)$, so it suffices to prove
the  result for the case of a linear group $\Gamma$. Moreover, we
can replace $F$ by the ring generated by the entries of the
generators of $\Gamma$ to deduce that $\Gamma$ is inside $\GL_n(A)$
for some finitely generated subring  $A$ of $F$.  Let now $G$ be
the Zariski closure of $\Gamma$.  If $G$ is virtually solvable
then $\Gamma$ has a finite index subgroup with an infinite
abelianization. This implies that for some $\ell,  \, \Gamma$ has
infinitely many $\ell$-dimensional irreducible representations and
so $\rho(\Gamma) = \infty$ and we are done.  So assume $G$ is not
virtually solvable and we can then replace $G$ by its quotient
modulo the solvable radical and replace $\Gamma$ by a finite index
subgroup (using  Corollary 4.5) to assume that $G$ is semisimple,
or even simple by taking a (non-trivial) simple quotient.

Now, we specialize $A$ into a global field $k$, keeping $\Gamma$
non-virtually solvable.  In fact, it was shown in \cite[Theorem~4.1]{LaLu}
that this can be done keeping $G$ as the Zariski
closure.

So, altogether we can assume  $\Gamma$ is a Zariski dense subgroup
in $G(k)$, where $G$ is a simple $k$-group.  Let $U_v$ denote the
closure, in the $v$-adic topology, of $\Gamma$ in $G(k_v)$ for
some non-archimedean place $v$ for which $G$ is isotropic over
$k_v$ and that closure is compact.  Note that all but finitely
many $v$ satisfy each condition, so there is no difficulty in
fixing $v$ satisfying both. By Pink's characterization of
Zariski-dense compact subgroups of semisimple groups over local
fields \cite{Pi}, there exists a finite extension $k'_v$ of $k_v$,
a simply connected, almost simple algebraic group $G'$ over
$k'_v$, and a compact open subgroup $U'_v\subset G'(k'_v)$ such
that $U_v$ is topologically isomorphic to the quotient of $U'_v$
by its intersection with the center of $G'(k'_v)$. Replacing
$U'_v$ with an open subgroup which meets that center only at the
identity, we see that $U_v$ has an open subgroup which is
topologically isomorphic to an open subgroup of the $k'_v$ points
of the almost simple algebraic group $G'$. Hence $\rho(\Gamma) \ge
\rho(U_v) \ge \frac 1{15}$ by Theorem~\ref{Isotropic}.
\end{proof}

We now show that Theorem~\ref{Linear-bound} is not valid in general
for finitely
generated, residually finite groups.  In fact, we can even prove:

\begin{theorem} There exists a finitely generated, residually
finite, infinite group $\Gamma$ with $\rho(\Gamma) = 0$.
\end{theorem}

\smallskip \noindent\emph{Proof}.  Let us recall first the result
of Liebeck and Shalev counting representations of the alternating
groups $A_k$.
\smallskip
\begin{theorem}[Liebeck-Shalev \cite{LiSh1}]
\label{Liebeck-Shalev}
For every $s>0$, $\lim\limits_{k \to
\infty} \calz_{A_k} (s) = 1$ where as before $\calz_{A_k} (s) =
\suml^\infty_{i = 1} r_i(A_k) i^{-s}$.
\end{theorem}

This theorem can be converted to an explicit upper bound on
representation growth, via the following lemma:

\begin{lemma}
\label{R-bound}
If $G$ is a perfect finite group, $0 < s < 1$,
and $\calz_G(s) < 1 + c$, then for every $n \in \bbn$, we have
$R_n(G) \le cn^{s} +1$.
\end{lemma}

\begin{proof}
As $G$ is perfect, $r_1(G) = 1$.
$$\left(R_n(G) - 1\right) n^{-s} = \suml^n_{i=2} r_i(G) n^{-s} \le
\suml^n_{i = 2} r_i(G) i^{-s} \le c,$$
which implies the lemma.
\end{proof}

Let us now recall some results of Segal \cite{S}: Let $\ell_0,
\ell_1, \ell_2, \dots $ be a sequence of positive integers.  We
construct, by induction, a sequence of finite groups $W_r$ as
follows:

\smallskip

$W_0 = A_{\ell_0}, \; W_1 = A_{\ell_1}^{\ell_0} \rtimes W_0, \; \;
W_2 = A_{\ell_2}^{\ell_0\ell_1} \rtimes W_1,\dots , W_r =
A_{\ell_r}^{\ell_0 \dots \ell_{r-1}} \rtimes W_{r-1}, \dots$.

\medskip

These are wreath products obtained as natural subgroups of the
automorphism group of the rooted tree with degree $\ell_0$ at the
origin and degree  $\ell_i + 1$ for all the vertices of level
$i>0$ (i.e. of distance $i$ from the origin). See [S] for the
detailed description.  Let $W$ be the profinite group $W =
\lim\limits_{\stackrel{\textstyle \longleftarrow}{\textstyle r}}
W_r$ with the obvious morphisms. It is also shown in [S] that $W$
contains a finitely generated dense subgroup $\Gamma$ whose
profinite completion is isomorphic to $W$ via the natural map
$\hat \Gamma \to W$ extending the embedding $\Gamma
\hookrightarrow W$.  It is easy to deduce that $\Gamma$ is not a
linear group.  Moreover, every representation of it factors
through one of the $W_r$.  Indeed, if $\Gamma$ had had an infinite
non-virtually solvable representation then (by an application of
strong approximation for linear groups \cite[pp.~389--407]{LS})
there would have been infinitely many simple groups of Lie type
among the composition factors of $\hat \Gamma = W$.  But as we
know, all the composition factors of $W$ are alternating groups.
Moreover, $W$ (and hence $\Gamma$) has the (FAb) property (i.e.,
every finite index subgroup has a finite abelianization) and so
$\Gamma$ has no infinite virtually-solvable quotients either. Thus
every representation factors through some $W_r$.  Moreover, as the
kernels $\ker (W\to W_r)$ are the only finite index normal
subgroups of $W$, a representation  of $W_r$ which does not factor
through $W_{r-1}$ must be faithful.

Let us now choose a sequence $\ell_0, \ell_1, \dots, \ell_r,
\dots$ which grows sufficiently fast.  More specifically we want

\begin{equation}
\label{big-lr}
\frac{\log |W_{r-1}|}{\log \ell_r} \, < \, \frac 1 r
\end{equation}
and
\begin{equation}
\label{small-zeta}
\calz_{A_{\ell_r}} \; (\frac 1r) < 1 + \frac{1}{L_{r-1}}
\end{equation}
where $L_{r-1} = \ell_0 \ell_1 \cdots\ell_{r-1}$

\medskip

Note that since $|W_{r-1}| = ( \half\ell_{r-1}!)^{\ell_0 \cdots
\ell_{r-2} }|W_{r-2}|$, the order of $W_{r-1}$ depends only on
$\ell_0,\ldots, \ell_{r-1}$, so we can choose $\ell_r$ large enough
to satisfy (\ref{big-lr}).  Also as $\calz_{A_k} (\frac 1r)
\underset{k\to\infty}{\longrightarrow} 1$ we can make sure that
$\ell_r$ also satisfies (\ref{small-zeta}).

Given the sequence $\ell_0, \ell_1, \dots,
\ell_r, \dots , $ let $W$ and $\Gamma$ be the groups as defined
before with respect to this sequence.  We have to bound
$r_n(\Gamma)$.

So given $n \in \bbn$, let $r $ be the unique natural number for
which $\ell_r -1 \le n < \ell_{r+1} -1$.  As $A_{\ell_{r+1}}$ is a
subgroup of $W_{r+1}$ and every non-trivial representation of
$A_{\ell_{r+1}}$ is of dimension at least $\ell_{r+1} - 1$, all
the $n$-dimensional representations of $\Gamma$ factor through
$W_r$ for this $r$.  By Proposition~\ref{Res-Ind},
\[
R_n(W_r) \le |W_{r-1}| \; R_n (A^{\ell_0\cdots
\ell_{r-1}}_{\ell_r} )
\]
where $R_n$ is the number of all irreducible representations of
dimension at most $n$.

Thus:
\begin{equation*}
 \lim\limits_{n\to\infty} \; \frac{\log R_n (\Gamma)}{\log n} \;
= \; \lim\limits_{n\to\infty} \; \frac{\log R_n(W_r)}{\log n} \;
\le  \lim\limits_{n\to\infty} \; \frac{\log |W_{r-1}|}{\log n} \,
+ \lim\limits_{n\to\infty} \; \frac{\log
R_n(A_{\ell_r}^{\ell_0\cdots\ell_{r-1}})} {\log n}
\end{equation*}
As $n\ge \ell_r-1$,  (\ref{big-lr}) implies the first summand is zero.

For the second summand, note that for
$L_{r-1} = \ell_0\cdots\ell_{r-1}$,
\[
\calz_{A_{\ell_r}^{L_{r-1}}} (s) = \calz_{A_{\ell_r}}
(s)^{L_{r-1}}
\]
Thus by (\ref{small-zeta}) we get

$$\calz_{A_{\ell_r}^{L_{r-1}} }(\frac 1r)
< \left(1+\frac{1}{L_{r-1}}\right)^{L_{r-1}} < e < 3.$$

This means by Lemma~\ref{R-bound} that
\[
R_n \left(A_{\ell_r}^{L_{r-1}}\right) \le 2n^{1/r} + 1\] Thus $
\lim\limits_{n\to\infty} \ \frac{\log R_n
(A_{\ell_r}^{L_{r-1}})}{\log n} \, = \, 0$ and so
\[ \rho(\Gamma) = \lim\limits_{n\to\infty} \; \frac{\log R_n
(\Gamma)}{\log n} \, = \, 0\] as promised. \hfill\qed

\section{Lattices in the same semisimple group}

The following theorem gives some support to our
Conjecture~\ref{LatticeComparison} which predicts the same
abscissa of convergence for lattices in the same semisimple
locally compact group.

Let $H = \prod\limits^\ell_{i = 1} \SL_2(K_i)$ where each $K_i$ is
a local field. Recall that $\rk H = \ell$, and when $\ell\ge 2$,
every irreducible lattice $\Gamma$ in $H$  is ($S$-)arithmetic. In
this case, Serre's conjecture \cite{Se} predicts that $\Gamma$ has
the CSP. This has been proved in the case of non-uniform lattices.
On the other hand, when $\ell = 1$, there are non-arithmetic
lattices, and even the arithmetic ones do not satisfy the CSP (see
\cite[Chapter~7]{LS} for an overview and references.) Here we
prove

\begin{theorem}
\label{power-of-SL2}
Let $H = \prod\limits^\ell_{i = 1} \SL_2(K_i)$ where
the $K_i$ are local fields of characteristic different than 2.
Let $\Gamma$ be an irreducible lattice of $H$.   Then:
\renewcommand{\theenumi}{\alph{enumi}}
\begin{enumerate}
\item If $\ell = 1$, then $\rho(\Gamma) = \infty.$
\item If $\ell\ge 2$ and $\Gamma$ has the CSP, then $\rho(\Gamma) = 2.$
\end{enumerate}
\renewcommand{\theenumi}{\alph{roman}}

\end{theorem}

Before proving the Theorem, let us make a few observations on the
connection between representation growth and subgroup growth of a
finitely generated pro-$p$ group $L$. As before, let $a_n(L)$
(resp. $s_n(L)$) be the number of subgroups of $L$ of index $n$
(resp. at most $n$) and $r_n(L)$ (resp. $R_n(L)$) the number of
irreducible representations of $L$ of degree $n$ (resp. at most
$n$).  For a finite index subgroup $M$ of $L$, denote
$$d(M) = \dim_{\bbf_p} (M/[M, M]M^p) = \log_p(|M/[M, M]M^p|)$$
and
$$e(M) = \log_p  (|M/[M, M]|).$$
Let
$$d_j(L)= \sup\{d(M)|[L:M]=p^j\},$$
$$e_j(L) =\sup\{e(M)| [L:M]=p^j\},$$
and
$$d^*_j(L) =\suml^j_{i=0} d_i(L).$$

\begin{proposition}
\label{compare-growth}
Let $L$ be a finitely generated pro-$p$ group and $j \in \bbn$.
Then:
\begin{enumerate}
\item[(a)] $p^{d_{j-1}(L)-1} \le a_{p^j}(L) \le p^{d^*_{j-1}(L)}$.
\item[(b)] $R_{p^j} (L) \ge \frac{1}{p^j} p^{d_j(L)}$.
\item[(c)] $\log R_{p^j}(L)\ge \frac1j \log a_{p^j}(L) -
\frac{j-1}{2}$.
\item[(d)] $r_{p^j}(L)\le a_{p^j}(L)\cdot e_j(L)$.
\end{enumerate}
\end{proposition}
\begin{proof}  (a) follows from \cite[Proposition 1.6.2]{LS}
while (b) follows from Proposition 4.4 above.  Now, by applying
(a) and then (b) we have:
\begin{equation*}
a_{p^j} (L) \le \prod\limits^{j-1}_{i=0} p^{d_i(L)} \le
\prod^{j-1}_{i=0} p^i R_{p^i} (L)
\end{equation*}
which gives (c).  Finally, (d) follows from the fact that a finite
$p$-group is an $M$-group (\cite{I}), i.e. every irreducible
representation of it of degree $p^j$ is induced  from a one
dimensional character of some subgroup of index $p^j$.
\end{proof}

\begin{corollary}
\label{fast-subgroups}
If the subgroup growth rate of $L$ is faster
than $n^{\log n}$\break (i.e. $\lim \sup \log s_n(L)/(\log n)^2 =
\infty)$ then $L$ does not have polynomial representation growth,
i.e. $\rho(L) = \infty$.
\end{corollary}

The Corollary follows from Part (c) of Proposition~\ref{compare-growth}.  We
should remark, that this Corollary is the best possible: it is
shown in \cite{LuMr} that $\SL_d(\bbf_p[[t]])$ (which is a
virtually pro-$p$ group) has polynomial representation growth,
while its subgroup growth is $n^{\log n}$ (see \cite[Chapter
4]{LS}).

Let us now use the above observations to treat the special case of
Theorem~\ref{power-of-SL2}(a) when $H=\SL_2 (\bbc)$ and $\Gamma$ a cocompact
lattice in $H$.  A well known conjecture, attributed to Thurston,
asserts that in this case, $\Gamma$ has a finite index subgroup
$\Delta$ which maps onto $\bbz$.   This would give our claim
immediately.  However, the conjecture remains wide open.  Still,
it was shown  in [Lu1] that such $\Gamma$ has a finite index
subgroup whose pro-$p$  completion $L$ is a Golod-Shafarevich
group (i.e. $d(L) \ge 4$ while $r(L) < d(L)^2/4$ where $r(L)$ is
the minimal number of pro-$p$ relations of $L$, i.e. $r(L) = \dim
H^2(L, \bbf_p)$).  For such groups, Shalev (cf. \cite[Theorem
4.6.4]{LS}) proved that for every $\e>0$, $a_n(L) \ge n^{(\log
n)^{2-\e}}$ for infinitely many integers $n$.  Thus
Corollary~\ref{fast-subgroups}
implies that $\rho(\Gamma) \ge \rho(L) = \infty$.

We mention in passing that Shalen and Wagreich (see [SW, Lemma
1.3]) proved a slightly better estimate on $d_j(\Gamma)$ (and
hence on $a_n(\Gamma)$).  A much better estimate was given
recently by Lackenby \cite{Ly}.

Now, to complete the proof of (a) of the theorem, we recall that
in all other cases, the analogue of Thurston's conjecture is true.
In fact, it is even known (by several different methods of proof; see
discussion in \cite[\S7.3]{LS}) that in all these cases
$\Gamma$ has a finite index subgroup which is mapped onto a
non-abelian free group.  Thus, clearly $\rho(\Gamma) = \infty$.

For (b), by Corollary~\ref{Commensurable} and
Proposition~\ref{Euler}, we may assume without loss of generality
that
\begin{equation}
\label{Factor} \calz_{\Gamma} (s) = \calz_{\G(\bbc)}
(s)^{\#S_\infty} \cdot \prod\limits_{v \notin S} \calz_{L_v} (s)
\end{equation}
where $G(\bbc)=\SL_2(\bbc)$ and
all but finitely many $L_v$ are of the form $\SL_2(\calo_v)$ where
$\calo_v$ is the ring of integers of the completion of the global
field $k$ at $v$, and the remaining $L_v$ are compact open
subgroups of groups which are either of the form $\SL_2$ of a
local field or $\SL_1$ of a quaternion algebra over a local field.
The Euler factors corresponding to these remaining factors have
abscissa of convergence $1$ by Theorems \ref{SL2}, \ref{SL1} and
\ref{not-char-2}. In determining whether $\calz_{\Gamma}(s)$ does or does not
have abscissa of convergence 2, they may therefore be omitted from
the Euler product. Likewise, $\calz_{\G(\bbc)} (s)$ has abscissa
of convergence 1 by Theorem~\ref{Archimedean}, so the first factor
on the right hand side of (\ref{Factor}) may be omitted from the
Euler product.  It remains to consider the abscissa of convergence
of
\begin{equation}
\label{OpenCurve}
\prod\limits_{v\notin T} \calz_{\SL_2(\calo_v)} (s)
\end{equation}
for some finite set of places $T$ of $k$.

As the $K_i$ are not of characteristic 2, the same is true for $k$ and therefore for the $k_v$.
For $s$ in the interval $[2,3]$, we have
$2^s\le 8$, and
\[(q+1)^{-s} < q^{-s} < (q-1)^{-s} \le 8 q^{-s}.\]
By Theorem~\ref{SL2}, for $q$ odd,
\begin{align*}
\calz_{\SL_2(\calo_v)}(s)
&> 1 + \left(q^{-s}+\frac{q-1}{2}(q-1)^{-s}\right)+
\frac{4q\left(\frac{q^2-1}2\right)^{-s}+\frac{q^2-1}2(q^2-q)^{-s}}{1-q^{1-s}} \\
&>1+\frac q2q^{-s} + \frac{\frac{q^2}2(q^2)^{-s}}{1-q^{1-s}} \\
&= 1 + \half q^{1-s} + \half (q^{1-s})^2(1-q^{1-s})^{-1} > (1-q^{1-s})^{-\half} \\
\end{align*}
In the other direction, we have
\begin{align*}
\calz_{\SL_2(\calo_v)}(s)
&< 1+q^{-s}+q^{1-s}+16q^{-s}+4q^{1-s}+128q^{-s}
+\frac{256q^{1-2s}+4q^{2-2s}+q^{2-2s}}{1-q^{1-s}} \\
&< 1 + 100q^{1-s} + \frac{1000 q^{2-2s}}{1-q^{1-s}}
< (1-q^{1-s})^{-100}.\\
\end{align*}
There are finitely many Euler factors for which $q$ is even (and
none at all if $k$ is of positive characteristic).  We may therefore assume
$q$ is odd for all Euler factors and prove
that $\calz_{\Gamma}(s)$ converges for $s>2$ and diverges
for $s=2$ by comparing the product (\ref{OpenCurve}) with
$\zeta_{k,T}(s-1)^{1/2}$ and $\zeta_{k,T}(s-1)^{100}$, where
$\zeta_{k,T}(s)$ is the usual Dedekind $\zeta$-function of $k$
with the Euler factors at $T$ removed (which is analytic for
$\Re(s) > 1$ and has a simple pole at $s=1$.) \qed

\section{Remarks and suggestions for further research}

Clearly, we are still at the qualitative stage in our understanding of
the abscissa of convergence for representation zeta-functions.  We
mention some of the questions left open by this paper.

For general finitely generated groups $\Gamma$,
are there any positive values
which cannot be achieved?  For infinite linear groups,
$\frac1{15}$ is probably not optimal.  A better understanding of
$\rho(U)$ where $U$ is a compact open subgroup of $E_8(k_v)$
seems likely to improve that value.  We do not even have a conjecture
regarding the greatest lower bound.

For arithmetic groups $\Gamma$ satisfying the congruence subgroup
property, we still lack a plausible conjecture for the value of
$\rho(\Gamma)$.  It is conceivable that without determining the
actual value, one can prove that $\rho(\Gamma)$ is always rational
in this setting.  We do not know if the values $\rho(\Gamma)$ as
$\Gamma$ ranges over arithmetic groups satisfying the CSP are
bounded above. By combining the results of \cite{LiSh2} with
upper bound estimates of the kind developed in Theorem~\ref{not-char-2}
likely that one can prove
$$\rho(\Gamma) \le c + \sup_v \rho(\Gamma_v),$$
where $\Gamma_v$ denotes the $v$-adic completion of $\Gamma$
and $c$ is an absolute constant.

This raises the question as to whether one can find reasonable
upper bounds for $\rho(U)$ for compact open subgroups $U\subset
\bg(K)$ of almost simple algebraic groups over non-archimedean
local fields. For instance, is there an absolute constant which
works for all $\bg$ and all $K$?  In a different direction, can
one prove equality for the values of $\rho$ for groups of fixed
type ($\SL_n$ for example), as $K$ ranges over local fields?
(Compare Theorem~\ref{SL2}, Theorem~\ref{SL1}, and \cite{LuNi}.) It is
conceivable that one could do so without being able to compute the
common value. As a step toward computing $\rho(\SL_n(\bbz_p))$, it
would be interesting to estimate the number of conjugacy classes
in $\SL_n(\bbz/p^r\bbz)$, for instance when $n$ and $p$ are fixed
and $r$ is allowed to grow.

One approach to these problems would be to try to imitate the
method of Theorem~\ref{SL1}.  Let $\bg$ be a group scheme of
finite type over the ring $\calo_K$ of  integers in a local field
$K$ with almost simple generic fiber.  Let $U = \bg(\calo_K)$ and
let $U_r$ denote the kernel of $U\to \bg(\calo_K/\pi_K^r)$. Every
element of $U/U_r$ lifts to a regular semisimple element of $U$.
Up to $\bg(K)$-conjugacy, there are finitely many maximal tori
$\bt_i$ in the generic fiber of $\bg$, and any regular semisimple
conjugacy class meets exactly one such maximal torus, and meets it
at finitely many points. The conjugacy classes of $U$ up to
$\bg(K)$-conjugacy are what gives rise to the general lower bound of
Proposition~\ref{ssimp-bound}.

Describing the regular semisimple conjugacy classes in $U$ (rather
than $\bg(K)$) brings the Bruhat-Tits building $\calb$ of $\bg$
over $K$ into the picture. (Note that for anisotropic groups, where
the building is trivial, Theorem~\ref{SL1}
says that Proposition~\ref{ssimp-bound} is
sharp.) For simplicity, let us suppose that $U$ is exactly the
stabilizer of a vertex  $x_0$ of the building. If, for example,
$g\in U_r$, then it fixes all the vertices in $B_{x_0}(r)$, the
ball of radius $r$ centered at $x_0$ in $\calb$. Now, if $h_i\in
\bg(K), i = 1, 2$,  and $h_i(x_0)\in B_{x_0}(r)$, then $h^{-1}_i g
h_i$ fixes $x_0$ and therefore lies in $U$. But $h_1^{-1} g h_1$
and $h^{-1}_2gh_2$ are not necessarily conjugate to each other in
$U$.  If $g$ is regular semisimple, then
$$u^{-1}(h_1^{-1}gh_1)u = h_2^{-1}gh_2,$$
is equivalent to
$h_2 u^{-1} h_1^{-1}\in Z_{\bg(K)}(g) = \bt(K)$, where $\bt$ is the
unique maximal torus containing $g$.  In other words, $h_2$
belongs to the double coset $\bt(K) h_1 U$, or, yet again,
$h_2(x_0)$ lies in the $\bt(K)$-orbit of $h_1(x_0)$. Thus,
counting torus orbits in balls in the building is closely
connected with the problem of classifying conjugacy classes in $U$
and thereby the problem of counting conjugacy classes in $U/U_r$.

It strongly suggests that when the building $\calb$ is ``larger,"
there are more conjugacy classes in $U$ (and $U/U_r)$ and
$\rho(U)$ tends to be larger.  As mentioned above, it is still not
clear if $\rho(U)$ can be arbitrarily large.  A good test case: is
$\rho(\SL_n(\bbz_p))$ bounded above independent of $n$?

\end{document}